\documentclass{amsart}

\usepackage{graphicx}
\usepackage{amssymb}
\usepackage{amsthm}
\usepackage{subcaption}

\newtheorem{theorem}{Teorema}
\newtheorem{lemma}[theorem]{Lemma}

\theoremstyle{definition}
\newtheorem{definition}[theorem]{Definition}
\newtheorem{ejem}[theorem]{Example}

\newcommand{\bd}{\partial}

\begin{document}

\title{Classification of genus two knots which admit a $(1,1)$-decomposition}
\author{M. Eudave-Mu\~noz}
\address{ \hskip-\parindent
Mario  Eudave-Mu\~noz\\
Instituto de Matem\'aticas\\
 Universidad Nacional Aut\'onoma de M\'exico\\
  M\'exico, D.F.\\ MX \\ and CIMAT, Guanajuato, 
 MX}
 \email{mario@matem.unam.mx}

\author{F. Manjarrez-Guti\'errez}
\address{ \hskip-\parindent
Fabiola Manjarrez-Guti\'errez \\
 Instituto de Matem\'aticas\\
 Universidad Nacional Aut\'onoma de M\'exico\\
 Cuernavaca, Morelos,
 MX}
\email{fabiola.manjarrez@im.unam.mx}

\author{E. Ram\'{i}rez-Losada}
\address{ \hskip-\parindent
Enrique Ram\'{i}rez-Losada\\
 Centro de Investigaci\'on en Matem\'aticas\\
 Guanajuato, GTO.
 MX}
\email{kikis@cimat.mx}

\date{\today}
\subjclass{57M25}
\keywords{$(1,1)$-knots, genus two knots, tunnel number one knots.}

\begin{abstract}
 We describe the genus two knots which admit a genus one, one bridge position. These are divided into several families, one consists of vertical bandings of two genus one $(1,1)$-knots, other consists of vertical bandings of two cross cap number two 2-bridge knots, and the last one consists of genus two tunnel number one satellite knots.
\end{abstract}

\maketitle
\section{Introduction}
\label{sec1}

A knot $K$ in $S^3$ admits a $(1,1)$-decomposition if there is a genus one Heegaard splitting $H_1\cup H_2$ of $S^3$ such that each of  $K \cap H_1$ and $K\cap H_2$ is a single arc that is parallel into the torus $T=\bd H_1=\bd H_2$; we also say that $K$ is a $(1,1)$-knot. This family contains all 2-bridge knots, all satellite tunnel number one knots, and it is contained  in the family of tunnel number one knots.  Morimoto and Sakuma \cite{MS}  and independently Eudave-Mu\~noz \cite{Eu} classified tunnel number one satellite knots. They have  a concrete description using a 4-tuple of integers. Goda and Teragaito \cite{GT} determined which of these knots have genus one, they conjectured that any non satellite genus one tunnel number one knot is a 2-bridge knot. 
Matsuda \cite{Ma} proved this conjecture for $(1,1)$-knots, then the conjecture is equivalent to the statement that any genus one tunnel number one knot admits a $(1,1)$-decomposition. Scharlemann settled it in \cite{Sc}, thus a genus one tunnel number one is either a 2-bridge knot or a satellite knot. Ram\'{i}rez-Losada and Valdez-S\'anchez \cite{RV} classified the family of crosscap number two tunnel number one knots which admit a $(1,1)$-decomposition. Moreover, they found tunnel number one knots bounding an essential once-punctured Klein bottle which are not $(1,1)$-knots. 

In this paper we consider the family of genus two $(1,1)$-knots. We divide the study of such knots into the satellite and non-satellite cases. For the non-satellite case we find that these knots can be described as a special banding of two genus one tunnel number one knots, or a special banding of two cross cap number two 2-bridge knots. In the case that the knot is satellite we determine the 4-tuple of the Morimoto-Sakuma construction that produces satellite genus two tunnel number one knots. We think that there are genus two tunnel number one knots which are not $(1,1)$.

The main technique used in the demonstrations is Morse theory, namely we consider the natural projection $h: T\times I \rightarrow I$, where $T$ is the Heegaard torus for $S^3$ and a genus two Seifert surface $F$ for a $(1,1)$-knot $K$. The surface $F$ can be isotoped in such a way that $F \cap S\times I$ possess useful properties to achieve the desire characterisation for the knot $K$. 
 
In Section \ref{sec2} we review $(g,b)$-decompositions for knots in $S^3$. In Section \ref{sec3}, we study the Morse position of a genus two Seifert surface for a $(1,1)$-knot. The classification for non-satellite knots is given in Section \ref{sec4}, see Theorems \ref{piecestheorem} and  \ref{nonsatcase}. Finally in Section \ref{sec5} we study the satellite case.

\section{$(g,b)$-decompositions}
\label{sec2}

A properly embedded arc $t$ in a solid torus $W$ is said to be \textit{trivial} if there is a disk $D$ in $W$ with $t\subset \bd D$ and $\bd D-t \subset \bd W$. Such a disk $D$ is called a \textit{cancelling disk}.

A knot or link $L$ in $S^3$ is said to be a \textit{$(g,b)$-knot} if there is a genus $g$ Heegaard splitting surface $S$ in $S^3$ bounding handlebodies $W_0$, $W_1$ such that, for $i=0,1$, $L$ intersects $W_i$ transversely in a trivial $b$-arc system. We assume that the point at infinity is contained in $W_1$. Consider $S\times I$ be a product regular neighbourhood of $S$ in $S^3$, and let $h: S\times I \rightarrow I$ be the natural projection map. We denote the level surfaces $h^{-1}(r)= S \times \{r\}$ by $S_r$ for each $0\leq r \leq 1$. $S_0$ bounds a handlebody  $H_0$, and $S_1$ bounds a handlebody $H_1$, such that $S^3= H_0 \cup (S\times I) \cup H_1$. Assume that $S_0 \subset W_0$, $S_1\subset W_1$, and that $h\vert (S\times I) \cap L$ has no critical points (so $(S\times I) \cap L$ consists of monotone arcs).

Let $F$ be a surface properly embedded in the exterior $E(L)=S^3 - int N(L)$. We say that $F$ is essential if it is incompressible, $\bd$-incompressible and not boundary parallel. 

By general position, an essential  surface can always be isotoped in $E(L)$ so that:

\begin{description}
\item[(M1)] $F$ intersects $S_0 \cup S_1$ transversely; we denote the surfaces $F \cap H_0$, $F \cap H_1$, $F \cap (S\times I)$ by $F_0$, $F_1$, $\tilde{F}$, respectively;
\item[(M2)] each component of $\bd F$ is either a level meridian circle of $\bd E(L)$ lying in some level set $S_r$ or it is transverse to all the level meridians circles of $\bd E(L)$ in $S \times I$;
\item[(M3)]for $i=0,1$, any component of $F_i$ containing parts of $L$ is a cancelling disk for some arc of $L \cap H_i$; in particular, such cancelling disks are disjoint from any arc of $L\cap H_i$ other than the one they cancel;
\item[(M4)] $h\vert \tilde{F}$ is a Morse function with a finite set $Y(F)$ of critical points in the interior of $\tilde{F}$, located at different levels; in particular, $\tilde{F}$ intersects each noncritical level surface transversely.
\end{description}

We define the complexity of any surface satisfying $(M1)-(M4)$ as the number
\begin{center}
$c(F)= \vert \bd F_0 \vert +  \vert \bd F_1 \vert + \vert Y(F) \vert$,
\end{center}

\noindent where $\vert Z \vert$ stands for the number of elements in the finite set $Z$, or the number of components of the topological space $Z$.

We say that $F$ is \textit{meridionally incompressible} if whenever $F$ compresses in $S^3$ via a disk $D$ with $\bd D= D \cap F$ such that $D$ intersects $L$ in one point interior to $D$, then $\bd D$ is parallel in $F$ to some boundary component of $F$ which is a meridian circle in $\bd E(L)$; otherwise, $F$ is \textit{meridionally compressible}. Observe that if $F$ is essential and meridionally compressible then a meridional surgery on $F$ produces a new essential surface in $E(L)$.

 Let $K$ be a $(1,1)$-knot in $S^3$, in this case the handlebodies $H_0$, $H_1$ are solid tori, and $S$ is a torus. A slope in $S$ is a class of isotopy of essential simple closed curves. Slopes in $S\times \{ x\}$ are in one to one correspondence with slopes in $S\times \{ y\}$, for any $x,y \in I$, so comparing curves on differents tori is accomplished via this correspondence. If $\beta$ and $\gamma$ are two slopes on $S$, $\Delta(\beta, \gamma)$ denotes, as usual, their minimal geometric intersection number. Denote by $\mu$ the essential simple closed curve in $S_0$ which bounds a meridional disk in $H_0$, and denote by $\lambda$ the essential simple closed curve in $S_1$ which bounds a disk in $H_1$. So $\Delta(\mu, \lambda)=1$. Thus a $(p_i,q_i)$-curve in $S_i$ represents a class of isotopy of essential simple closed curves $p_i \mu + q_i \lambda$ in $S_i$.

 The following lemma can be found in \cite{RV}, where it is stated for spanning surface. Here we consider only spanning orientable surfaces, that is Seifert surfaces, so the original statement is slightly reduced.
 
\begin{lemma}
\label{lem1}
Suppose $K$ is a $(1,1)$-knot and $K$ is not a torus knot. Let $F'$ be a Seifert surface for $K$ in $S^3$ such that $F = F' \cap E(K)$ is essential in $E(K)$. If $F$ is isotoped so as to satisfy (M1)-(M4) with minimal complexity, then $\vert Y(F)\vert= 1-\chi(F)$, and
\begin{enumerate}
\item each critical point of $h\vert \tilde{F}$ is a saddle,
\item for $0\leq r \leq 1$ any circle component of $S_r \cap F$ is nontrivial in $S_r - K$ and in $F$, and not parallel in $F$ to $\bd F$,
\item for $i=0,1$, $F_i$ consists of one cancelling disk and a collection of disjoint annuli each having boundary slope $(p_i, q_i)$ in $S_i$ with $\vert q_0 \vert, \vert p_1\vert \geq 2$. Each solid torus cobounded by an annulus in $F_i$ and the annulus in $S_i$, for $i=0,1$,  contains the corresponding cancelling disk, and
\item the saddle closest to either the 0-level or 1-level does not join circle components.
\end{enumerate}
\end{lemma}

We will agree that the saddle closest to the 0-level is the first saddle, and  the saddle closest to the 1-level is the last saddle, the saddles in between will be called second, third, and so on.

Assuming the hypothesis of  Lemma \ref{lem1} we have the following lemma:

\begin{lemma}
\label{lem2} The first saddle point joins the arc, corresponding to the cancelling disk, with itself, changing it into one arc and one essential simple closed curve. The last saddle point, when read backwards does the same.
\end{lemma}
\begin{proof}
Consider the first saddle. Suppose it joins the arc component $\alpha_0$ of  $F \cap S_0$ to a circle component $\beta$ of $F \cap S_0$, and $\beta$  corresponds to one boundary component of an annuli $A$ component of $F_0$.  An arc component $\gamma$ is created in a level slightly above the saddle level. Pushing down the saddle slightly below level 0 isotopes $F$ in such a way that the annulus $A$ glued with the cancelling disk along the arc $\alpha_0$ becomes a component $A'$ of $F_0$ which intersects $S_0$ in an arc $\gamma$ and a curve $\lambda = \bd A - \beta$. 
Pushing up a regular neighborhood of a spanning arc in $A'$ with end  points in $\beta$ and $\lambda$
isotopes $F$ to a surface that satisfies (M1)-(M4), with the same number of critical points but the number $\vert \bd F_0\vert$ decreases by 2, contradicting the assumption that $c(F)$ is minimal. 

Suppose the first saddle changes one arc  into one arc $\alpha$ and one trivial simple closed curve $c$; by $(2)$ of Lemma \ref{lem1}, $c$ is not trivial in $F$. There is a disk $D$ bounded by $c$ on $S\times \{ r\}$, for some $r \in (0,1)$. If $int(D)$ does not contain the arc $\alpha$ then $D$ is a compressing disk for the surface $F$, which is not possible. If $\alpha$ is contained in $int(D)$, then we can find disk $D'$ with boundary  $c$  such that $D' \cup D$ bounds a 3-ball that contains the cancelling, thus $D'$ is a compression disk for $F$ which is a contradiction.

The proof for the case of the last saddle is similar.
\end{proof}

\section{Morse position for genus two surfaces for $(1,1)$-knots}
\label{sec3}

\begin{definition} An essential properly embedded annulus $A$ in $S\times I$ is called a \textit{spanning annulus for a $(1,1)$-knot $K$}, if $A$ can be isotoped to be disjoint from $K$ and its boundary slope in $S_0$ is of the form $(p_0,q_0)$ for some $\vert p_0\vert, \vert q_0 \vert \geq 2$.
\end{definition}

\begin{lemma}
\label{lem4}
Let $K$ be a $(1,1)$-knot in $S^3$. If there is  a spanning annulus for $K$, then $K$ is either the trivial knot, or a $(p,q)$-torus knot or $K$ is a satellite of a $(p,q)$-torus knot.
\end{lemma}
\begin{proof}
Let $A$ be a spanning annulus for $K$ in $S \times I$, with boundary slope $(p,q)$,   $\vert p\vert, \vert q \vert \geq 2$. Cut $S\times I$ open along $A$ to obtain a solid torus; we can identify this solid torus with  $V\times I$ where $V$ an annulus, with core the $(p,q)$-torus knot. After cutting, the knot $K$ is in 1-bridge position with respect to $V$. For every $t \in I$, the knot $K$ intersects $V \times \{ t\}$ in two points. If $K$ has wrapping number equal to $0$, then there is a compressing disk $D$ for  $V \times I$ disjoint from $K$. After an isotopy we can assume that $D$ is isotopic to $\alpha \times I$, where $\alpha$ is an arc properly embedded  in $V$ with endpoints in different boundary components of $V$ such that $K$ remains in 1-bridge position with respect to $V$. Cutting along $\alpha \times I$ we obtain $E\times I$, where $E$ is a disk, and the knot $K$ in 1-bridge position with respect to $E$. This implies that $K$ in the trivial knot. If $K$ has wrapping number $\geq 1$, the knot $K$ is either a $(p,q)$-torus knot or a satellite of a $(p,q)$ torus knot.
\end{proof}

From now on we will assume that $K$ is a genus two knot. For a  $(p,q)$ torus knot it is known that its genus is given by $\frac{1}{2}(p-1)(q-1)$, thus if $K$ is a genus two torus knot then it must be the $(5,2)$ torus knot. Therefore, we can assume further that $K$ is not a torus knot.

\begin{lemma}
\label{lem6}
Let $K$ be a genus two  $(1,1)$-knot in $S^3$, and assume there is no spanning annulus for $K$ in $S\times I$. Let $F$  be a genus two Seifert surface for $K$ which has been isotoped to satisfy (M1)-(M4) with minimal complexity. Then $F \cap S_i$ has at most two circle components for $i=0,1$.
\end{lemma}

\begin{proof}
Lemma \ref{lem1} implies that $F_0$, $F_1$ consist of one cancelling disk and a collection of disjoint annuli. Thus $F  \cap (S_0 \cup S_1)$ consists of an even number of circle components. Notice that $\vert Y(F) \vert =4$ and by Lemma \ref{lem2} the first and fourth saddles join an arc with itself, creating two new circle components. Let first assume that $F \cap S_0$ contains at least six circle components and one arc, and $F\cap S_1$ contains $2n$ curves and one arc. Thus, a level  $S_r$ slightly above the first saddle contains at least seven circle components and one arc, and a level $S_{r'}$ slightly below the fourth saddle  contains $2n+1$  circle components and one arc. We are left with two saddles to connect the circles to complete the genus 2 surface.  After the third saddle occurs we are left with one arc and at least five curves, which have to be connected with the arc and the $2n+1$ curves in $S_{r'}$. If these two sets of curves have different cardinality, the surface can not be completed. If the cardinalities agree, then at least one circle component of $F \cap S_0$ must flow along an annulus component of $\tilde {F}$ from $S_0$ to $S_1$ without interacting with the saddles. Thus $\tilde{F}$ has at least one annulus component which, by Lemma \ref{lem1}, has boundary slope of the form $(p,q)$ in $S_0$ for some $\vert p \vert, \vert q \vert \geq 2$, and must be a spanning annulus for $K$. Thus this case is not possible.

Now suppose that  $F \cap S_0$ contains four circle components and one arc, and $F\cap S_1$ contains at most four circle components and one arc. A level  $S_r$ slightly above the first saddle contains five circle components and one arc, and a level $S_{r'}$ slightly below the last saddle contains at most five circle component and one arc. If the second saddle connects two circle components, we will have a trivial curve in some level, which bounds a disk containing the arc, and then one of the curves used is the new curve created after the first saddle. Then the third saddle must eliminate this trivial curve, so the third saddle connects the trivial curve with the arc or with another circle. So, at most 3 circles were involved, including the one created after the first saddle. Then there are 2 circles in $S_0$, which are unperturbed. 
If the second saddle connects a circle with itself, we will have a trivial curve in some level, which bounds a disk containing the arc, and then the third saddle must eliminate this trivial curve, so the third saddle connects the trivial curve with the arc or with another circle. Again, there are 2 circles in $S_0$, which are unperturbed. The other possibility is that the second and third saddles connect the arc with two circles, but again, there are at least 2 unperturbed circles. In any case, the two unperturbed circles are joined to upper circles. So at least one circle component of $F \cap S_0$ must flow along an annulus component of $\tilde {F}$ from $S_0$ to $S_1$ without interacting with the saddles. Thus $\tilde{F}$ has at least one annulus component which, by Lemma \ref{lem1}, has boundary slope of the form $(p,q)$ in $S_0$ for some $\vert p \vert, \vert q \vert \geq 2$, and must be a spanning annulus for $K$.  Therefore $F \cap S_i$ has at most two circle components, for $i=0,1$.

\end{proof}

Given a $(1,1)$-knot $K$ and a genus two Seifert surface $F$ for $K$, assume $F$ has been isotoped so as to satisfy (M1)-(M4) with minimal complexity. By Lemma \ref{lem1} the surface $\tilde{F}$ contains exactly four saddle points. The saddles are of the following types:

\begin{description}
\item [\textbf{Type 1}] A saddle changing an arc into one arc and one simple closed curve.
\item [\textbf{Type 2}] A saddle changing an arc and one simple closed curve into an arc.
\item [\textbf{Type 3}] A saddle changing one simple closed curve into two simple closed curves.
\item [\textbf{Type 4}] A saddle changing two simple closed curves into one simple closed curve.
\end{description}

According to Lemma \ref{lem2}, the first saddle is of type 1, and the fourth saddle of type 2.
There are sixteen possible cases for the second and third saddle. Some cases can be discarded according to the following lemmas:

\begin{lemma}
\label{type4saddle}
The third saddle is not of type $3$.

\end{lemma}
\begin{proof}
If the third saddle is of type 3, it joins one curve with itself, changing it into an essential curve and a trivial curve such that the arc is contained in the disk bounded by the trivial curve. So when the fourth saddle occurs, it must join and arc with a trivial curve, which contradicts Lemma \ref{lem2}.   
\end{proof}

\begin{lemma}
The second saddle is not of type $4$.
\label{secondisnottype4}
\end{lemma}
\begin{proof}
Suppose the  second saddle is of type 4, then two level curves are joined into a single curve. By Lemma \ref{lem2}, all the curves just before the second saddle are essential. By Lemma \ref{lem1}, one of the curves is the one created by the first saddle. The surface $\tilde{F}$ can be isotoped in such a way that the first and second saddles are at the same level and we can interchange them. Now the first saddle joins an arc with a curve, contradicting Lemma \ref{lem2}.
\end{proof}

According to the lemmas above, we have  the possibilities shown in Table \ref{tabla2}.

  \begin{center}
 \begin{table}
 \begin{tabular}{| c | c| c |  }
 \hline
 case  & 2nd. saddle& 3rd. saddle  \\ \hline
 1  & type 1& type 1 \\ \hline
 2 & type 1& type 2 \\ \hline
 
 3 & type 1& type 4 \\ \hline
 4 & type 2& type 1 \\ \hline
 5  & type 2& type 2 \\ \hline

 6 & type 2& type 4 \\ \hline
 7 & type 3& type 1 \\ \hline
8 & type 3& type 2 \\ \hline

9  & type 3& type 4 \\ \hline

 \end{tabular}
\caption{}
 \label{tabla2}
 \end{table}
 \end{center}

\begin{lemma}
\label{silla3tipo4}
If third saddle is of type $4$ it cannot join two essential curves into a single curve.
\end{lemma}
\begin{proof}
Suppose the two curves are essential and  the arc at that regular level is contained in the annulus bounded by the curves. After the third saddle occurs a trivial curve surrounds the arc, then the fourth saddle joins the arc with the trivial curve, which is not possible, by Lemma \ref{lem2}.
\end{proof}

\begin{lemma}
\label{lem12}
If third saddle is of type $1$ it cannot change one arc into one arc and a trivial curve.
\end{lemma}
\begin{proof}
If the third saddle of type 1 changes one arc into an arc and a trivial curve, the new arc must be contained in the interior of the disk bounded by the trivial curve. Then the fourth saddle must join the arc with a trivial curve, this contradicts Lemma \ref{lem2}.
\end{proof}

 Let us denote by $F'_0$, $F_1'$ the components of $F_0$, $F_1$, respectively, other than the cancelling disks. Let us assume that  $K$ is neither a torus knot nor a satellite knot, by Lemma \ref{lem4} the product $S\times I$ contains no spanning annuli for $K$. Moreover, Lemma \ref{lem6} implies that $F'_i$ is an annulus or empty, for $i=0,1$. Without loss of generality, $F_0'$ and $F_1'$ fit in one of the following cases:
 
\begin{description}
\item [\textbf{Case A}] Both $F_0'$ and $F_1'$ are annuli.
\item [\textbf{Case B}] $F_0'$ is an annulus and $F_1'$ is empty.
\item [\textbf{Case C}] Both  $F_0'$ and $F_1'$ are empty.
\end{description}

 We will analyze the nature of the second and third saddle for each of the cases (A), (B), and (C). Each subcase will be denoted by \textbf{Ai, Bi} and \textbf{Ci}, with $i=1,...,9$, according to Table \ref{tabla2}.

 \textbf{Case A:} Both $F_0'$ and $F_1'$ are annuli. In this case $F\cap (S\times \{ i\})$, for $i=0,1$, consists of two non-trivial curves and one arc. Because of this, we can interchange the roles of the first and fourth saddle, and the second and third saddle. This observation will reduce the analysis of nine cases to six cases. Cases 1 and 5, cases 3 and 8, and cases 6 and 7 are equivalent. Thus the remaining cases are: 1, 2, 3, 4, 6 and 9.
 
 Next we count the number of curves which are  produced or eliminated by each of the saddles. We start with two curves and one arc. After passing through  the four saddles we have to end with two curves and one arc. 
 
 Cases 1 and 6 can not occur, since the last one eliminates all the curves and the former creates four curves.
 
 Thus the only possible cases are 2, 3, 4, and 9.
 
 Let $\alpha$, $b$ and $c$ be the arc and two essential closed curves, respectively,  in  $F\cap (S\times \{ 0\})$.  The first saddle, of type 1, changes $\alpha$ into an arc $\alpha_1$ and one essential curve $d$. Indeed $d$ is parallel to $b$ and $c$. Without loss of generality, we can assume that the arc $\alpha_1$ is contained in the annulus bounded by $b$ and $d$ on $S\times \{ \epsilon\}$, for some $\epsilon$-level slightly above the first saddle. 

\textbf{Case A2:} Second saddle of type 1 and third saddle of type 2. 

The second saddle changes $\alpha_1$ into an arc $\alpha_2$ and one curve $e$. 
If the curve $e$ is inessential, thus in a regular level slightly above the second saddle, $e$ bounds a disk that contains $\alpha_2$. The third saddle of type 2 joins the arc $\alpha_2$ and $e$ into one arc $\alpha_3$, finally the fourth saddle joins $\alpha_3$ either with $d$ or $b$.  In both cases the curve $c$ flows along an annulus component $A$ of $\tilde{F}$ from $S_0$ to $S_1$, the boundary slope of such annulus is $(p_0,q_0)$ in $S_0$ with $ \vert q_0 \vert \geq 2$, and is also $(p_1,q_1)$ in $S_1$, with $ \vert p_1 \vert \geq 2$, but $(p_0, q_0)= (p_1, q_1)$ implying  that $A$ is an spanning annulus for $K$, see Figure \ref{casoA2:sfig2} . Lemma \ref{lem4} implies that $K$ is a satellite knot, but this is not possible.

If the second saddle changes $\alpha_1$ into an arc $\alpha_2$ and one essential curve $e$.  There are two possibilities, the arc $\alpha_2$ lies in the annulus region between $b$ and $e$, or in the annulus region between $d$ and $e$. 

Suppose now that the arc $\alpha_2$ lies in the annulus region between $d$ and $e$. Then the third saddle changes $\alpha_2$ and $d$ (or $e$) into an arc $\alpha_3$. The fourth saddle then changes $\alpha_3$  and $e$ or $c$ ($b$ or $d$) into an arc $\alpha_4$. Then the curve $b$ ($c$) flows from $S_0$ to $S_1$, as above, we can find a spanning annulus for $K$,  which is a contradiction, see Figure \ref{casoA2:sfig3}.

Suppose then that the arc $\alpha_2$ lies in the annulus region between $b$ and $e$. If the third saddle changes $\alpha_2$ and $e$ into an arc $\alpha_3$, then the fourth saddle changes $\alpha_3$ and $b$ or $d$ into an arc $\alpha_4$, but then $c$ flows from $S_0$ to $S_1$, implying that $K$ is a satellite knot, since there would be a spanning annulus, as before, but this is not the case. The reader can picture this case using the one above. Therefore the third saddle must change $\alpha_2$ and $b$ into an arc $\alpha_3$, and the fourth saddle changes $\alpha_3$ and $c$ into an arc $\alpha_4$, for if it changes $\alpha_3$ and $d$, $c$ will flow from $S_0$ to $S_1$.

Summarizing, the second saddle changes $\alpha_1$ into an arc $\alpha_2$ and one essential curve $e$, such that $\alpha_2$ lies in the annulus region between $b$ and $e$.
The third saddle changes  $\alpha_2$ and $b$ into an arc $\alpha_3$, and the fourth one changes $\alpha_3$ and $c$ into an arc $\alpha_4$. See Figure \ref{casoA2:sfig4}. Observe that in this situation, the curve $c$ flows all the way from level 0 to a regular level $r$ below  the fourth saddle, defining an annulus $A_1$. From level $r$ to level 1, the curve $d$ flows and defines an annulus $A_2$. Since $c$ and $d$ have the same slope  we can glue $A_1$ and $A_2$ along the boundary on level $r$, obtaining an annulus $A$ disjoint from $K$. The slope of $A$ coincides with the  slope of $c$, $(p_0, q_0)$; and the slope of  $d$, $(p_1, q_1)$. Thus we have $(p_0,q_0)=(p_1, q_1)$ and, by Lemma \ref{lem1},  $\vert q_0 \vert \geq 2$ and  $\vert p_1 \vert \geq 2$, therefore $A$ is a spanning annulus for $K$, which is a contradiction. Hence this case is not possible.

\begin{figure}
\numberwithin{figure}{section}
\begin{subfigure}{.8\textwidth}
  \centering
  \includegraphics[width=.8\linewidth]{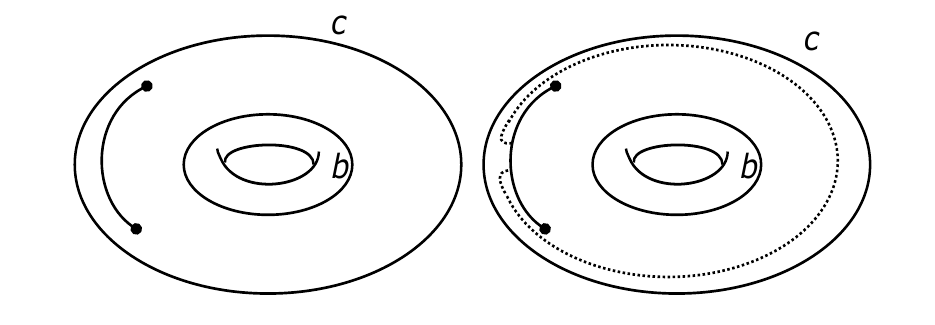}
  \caption{Initial condition}
  \label{casoA2:sfig1}
\end{subfigure}%
\hfill
\begin{subfigure}{.8\textwidth}
  \centering
 \includegraphics[width=1.5\linewidth]{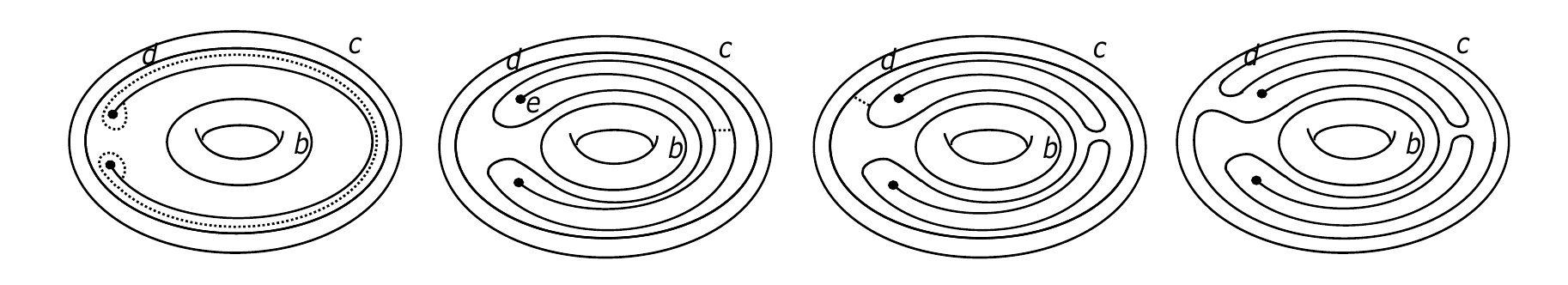}
 \caption{}
 \label{casoA2:sfig2}
\end{subfigure}
\hfill
\begin{subfigure}{.8\textwidth}
 \centering
 \includegraphics[width=1.5\linewidth]{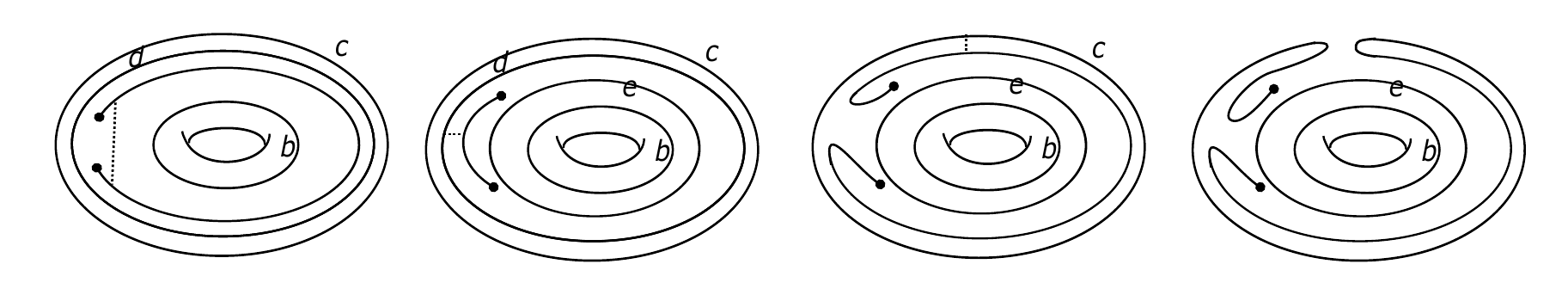}
 \caption{}
\label{casoA2:sfig3}
\end{subfigure}%
\hfill
\begin{subfigure}{.8\textwidth}
  \centering
  \includegraphics[width=1.5\linewidth]{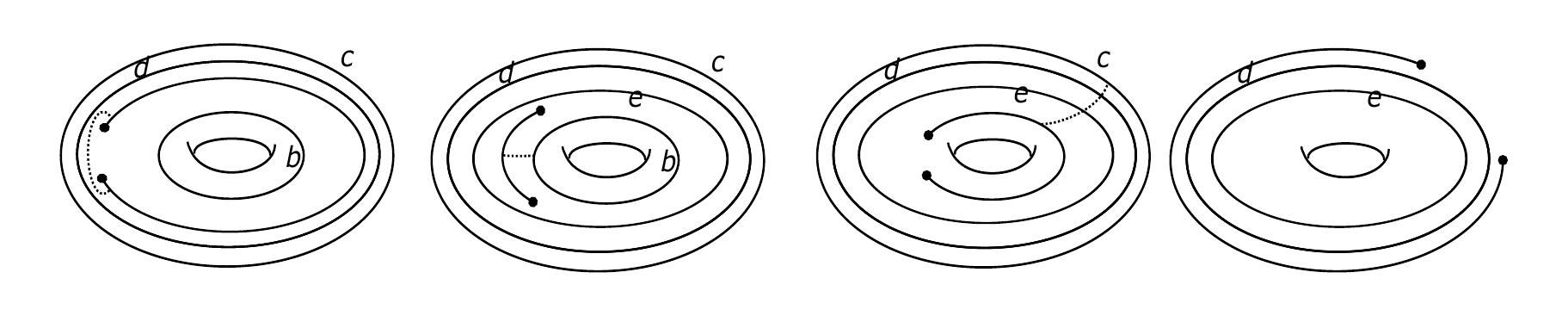}
 \caption{}
 \label{casoA2:sfig4}
\end{subfigure}
\caption{Case A2: The doted lines represent the saddles.}
\label{casoA2}
\end{figure}

 \textbf{Case A3:} Second saddle of type 1 and third saddle of type 4. 
 
 This case is similar to Case A2, up to the third and fourth saddle. Thus, above the second saddle, in a regular level $S_r$ we have one arc $\alpha_2$ and four curves $b, c,d,e$, where the first three are essential.
 If $e$ is essential as well, there are two possibilities, the arc $\alpha_2$ lies in the annulus region between $b$ and $e$, or in the annulus region between $d$ and $e$. The third saddle of type 4, changes either $b$ and $e$ or the curves $d$ and $e$ into a trivial curve, which contradicts Lemma \ref{silla3tipo4}.
 If $e$ is inessential, then the arc $\alpha_2$ is contained in the disk bounded by $e$. The third saddle of type 4, must join $e$ with either $b$ or $d$. If $e$ and $d$ are joined by the third saddle, a essential curve $f$ is created and the arc $\alpha_2$ lies in the annulus region between $f$ and $c$, the last saddle joins $\alpha_2$ with $f$ or with $c$, in either case the curve $b$ flows from level 0 to leve 1, thus we obtain a spanning annulus. 
 On the other hand, if the third saddle joins $e$ with $b$, then the arc $\alpha_2$ is contained in the annulus region by the curve $f$, created by the third saddle, and the curve $c$. Thus the fourth saddle joins the arc with $f$ or $c$. As in the last part of case $A_2$, we can create a spanning annulus for $K$. This case is impossible.

\textbf{Case A4:} Second saddle of type 2 and third saddle of type 1.

The second saddle may join $\alpha_1$ with $b$ or $d$. If the second saddle changes the arc $\alpha_1$ and the curve $d$ into one arc $\alpha_2$, then the third saddle will change $\alpha_2$ into $\alpha_3$ and an essential curve $e$, essentially is guaranteed by Lemma \ref{lem12}.  The fourth saddle will connect $\alpha_3$ with one of the curves $b$, $c$ or $e$. But in any case, $b$ or $c$ will flow from $S_0$ to $S_1$, which allow us to construct a spanning annulus, obtaining a contradiction.

So, suppose that the second saddle changes the arc $\alpha_1$ and the curve $b$ into one arc $\alpha_2$. The third saddle is of type 1 changing  the arc $\alpha_2$ into an arc $\alpha_3$ and an essential curve $e$. The fourth saddle connects $\alpha_3$ with the curve $e$, $d$ or $c$. In the first two cases $c$ will flow from $S_0$ to $S_1$, which is not possible since we can obtain a spanning annulus. Thus the fourth saddle joins $c$ and $\alpha_3$, in this situation, the curve $c$ flows all the way from level 0 to a regular level $r$ below  the fourth saddle, defining an annulus $A_1$. From level $r$ to level 1, the curve $d$ flows and defines an annulus $A_2$. Since $c$ and $d$ have the same slope  we can glue $A_1$ and $A_2$ along the boundary on level $r$, to obtain an annulus $A$ disjoint from $K$. By the argument used in case \textbf{A2}, it follows that $A$ is a spanning annulus for $K$. Hence case $A4$ cannot happen

\textbf{Case A9:} Second saddle of type 3 and third saddle of type 4.

The second saddle of type 3 changes one simple curve into two simple curves. There are two possibilities, either the saddle joins $b$ or $d$ with itself, for otherwise we get a compressing disk. Suppose that the saddle changes the curve $b$ into an essential curve $e$ and a trivial curve $f$ such that the arc $\alpha_1$ is contained in the disk bounded by $f$.  The annulus cobounded by $e$ and $c$ contains $f$ and $\alpha_1$. The third saddle changes two curves into one curve, Lemma \ref{silla3tipo4} implies that this saddle cannot join two essential curves, therefore there are two possibilities, $f$ and $e$ are joined, or $f$ and $c$ are joined. In the first situation, the curve $c$ flows from level 0 to level 1, giving rise to a spanning annulus. If $f$ and $c$ are joined, then the curve $c$ flows all the way from level 0 to a regular level $r$ below  the fourth saddle, defining an annulus $A_1$. From level $r$ to level 1, the curve $d$ flows and defines an annulus $A_2$. Since $c$ and $d$ have the same slope  we can glue $A_1$ and $A_2$ along the boundary on level $r$, we obtain an annulus $A$ disjoint from $K$. By the argument used in case \textbf{A2}, it follows that $A$ is a spanning annulus for $K$.

The other possibility is that the second saddle changes the curve $d$ into an essential curve $e$ and a trivial curve $f$ such that the arc $\alpha_1$ is contained in the disk bounded by $f$. The above argument shows that this case is impossible.

\textbf{Case B:} $F_0'$ is an annulus and $F_1'$ is empty. In this case $F\cap (S\times \{ 0\})$ consists of two non-trivial curves and one arc; and $F\cap (S\times \{ 1\})$ is just one arc. Again by counting the number of curves and arcs, we observe that the only possibilities are cases 5 and 6, from Table \ref{tabla2}. 

Let $\alpha$, $b$ and $c$ be the arc and two essential closed curves, respectively,  in  $F\cap (S\times \{ 0\})$.  The first saddle, of type 1, changes $\alpha$ into an arc $\alpha_1$ and one essential curve $d$. Indeed $d$ is parallel to $b$ and $c$. Without loss of generality, we can assume that the arc $\alpha_1$ is contained in the annulus bounded by $b$ and $d$ on $S\times \{ \epsilon\}$, for some $\epsilon$-level slightly above the first saddle.

\textbf{Case B5:} Second saddle of type 2 and third saddle of type 2.

Just above the first saddle there are 3 curves, $b$, $d$ and $c$, and the arc $\alpha_1$ lies between $b$ and $d$. The next 3 saddles are of type 2, and then the 3 curves are eliminated in certain order. The possibilities are: $d$, $b$, $c$, or $d$, $c$, $b$, or $b$, $d$, $c$, or  $b$, $c$, $d$. 

All these cases are possible and in all cases an essential annulus can be constructed, since the knot is not satellite such an  annulus can not be spanning for $K$.

Let us describe how we construct an essential annulus for the case $d$, $b$, $c$, for the other cases the essential annulus in constructed in a similar fashion.

Observe that from the 0-level to a regular level slightly below the fourth saddle the curve $c$ flows determining an annulus component $A$ of $\tilde{F}$. At fourth critical level we see that the arc $\alpha_3$, created by the third saddle, hits the curve c on some side of a small regular neighbourhood $N(c)$ on that level. Let us assume that the curve $c$ is touched on the $+$-side by the fourth saddle and let $c' \in N(c)$ be a copy of $c$ lying on the $-$-side of $c$, see right hand side of Figure \ref{B5:sfig3} . The annulus $A$ can be enlarged to an essential annulus $A'$ with boundary slope equal to the slope of $c$. Since $K$ is non satellite, the slope $(p_0, q_0)$ satisfies $\vert q_0 \vert \geq 2$ and  $\vert p_0\vert=1$.

\begin{figure}
\begin{subfigure}{.8\textwidth}
  \centering
  \includegraphics[width=.8\linewidth]{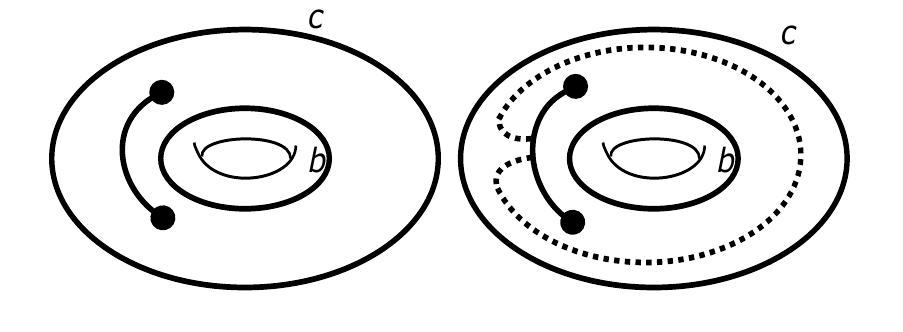}
  \caption{}
  \label{B5:sfig1}
\end{subfigure}%
\hfill
\begin{subfigure}{.8\textwidth}
  \centering
  \includegraphics[width=.8\linewidth]{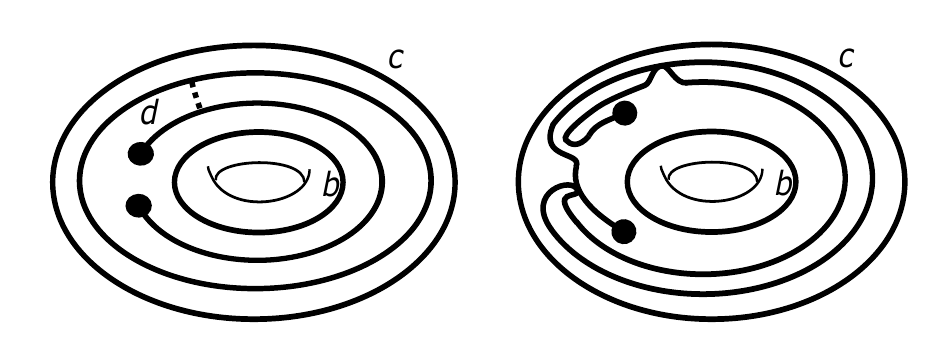}
  \caption{}
  \label{B5:sfig2}
\end{subfigure}
\hfill
\begin{subfigure}{.8\textwidth}
  \centering
  \includegraphics[width=1.2\linewidth]{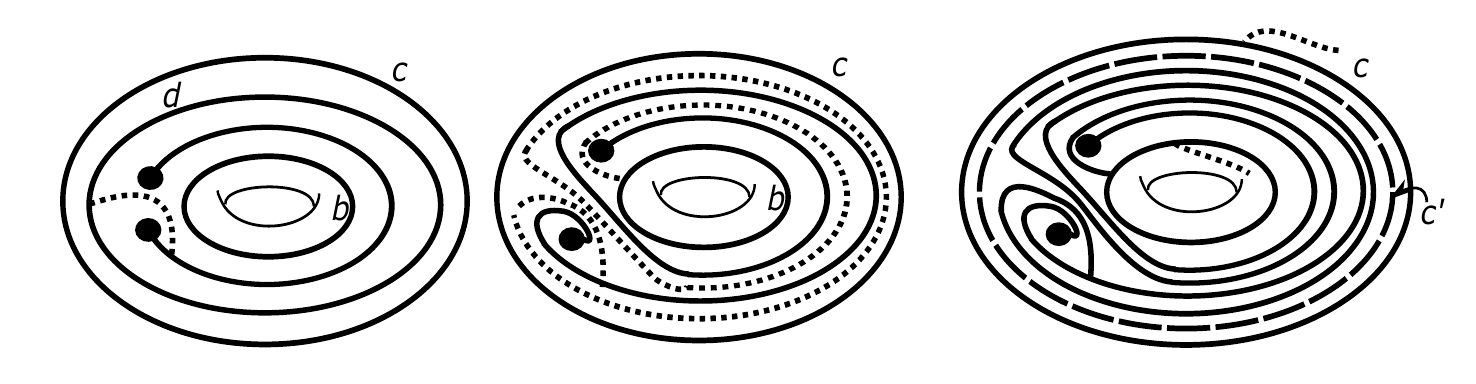}
  \caption{}
\label{B5:sfig3}
\end{subfigure}%

\caption{Case B5}
\label{B5}
\end{figure}

\textbf{Case B6:} Second saddle of type 2 and third saddle of type 4.

As above, after the second saddle occurs we have the curves $c$, $d$ (or $b$ and $c$) and the arc $\alpha_2$.
The third saddle changes the curves $c$ and $d$ (or $b$ and $c$) into a trivial curve, contradicting Lemma \ref{silla3tipo4}, therefore this is not possible.

 \textbf{Case C:} Both  $F_0'$ and $F_1'$ are empty. In this case $F\cap (S\times \{ i\})$, for $i=0,1$, consists of one arc. Similarly to case (A), the nine possibilities are reduced to six, and by counting arcs and curves we conclude that the only cases are: 2, 3, 4 and 9.
 
  Let $\alpha$ be the arc   in  $F\cap (S\times \{ 0\})$.  The first saddle, of type 1, changes $\alpha$ into an arc $\alpha_1$ and an essential curve $b$.

\textbf{Case C2:} Second saddle of type 1 and third saddle of type 2.

\begin{figure}
\numberwithin{figure}{section}
\begin{subfigure}{.8\textwidth}
  \begin{center}
  \includegraphics[width=.8\linewidth]{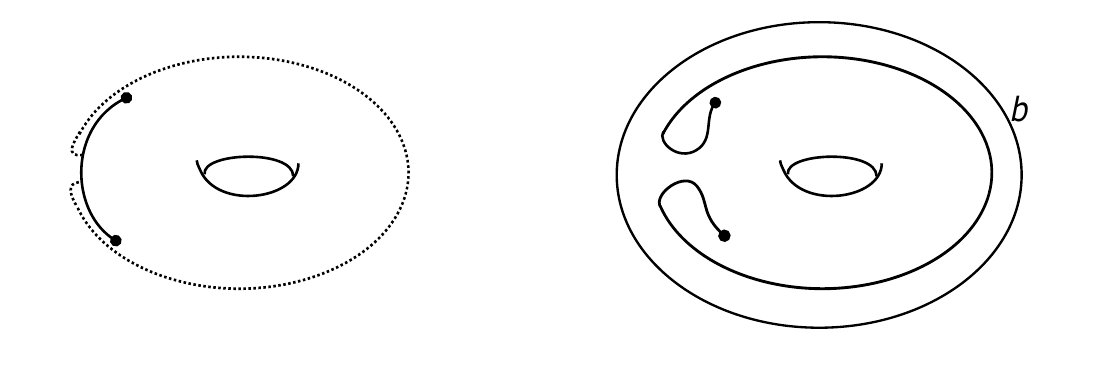}
  \end{center}
  \caption{Initial condition}
  \label{casoC2:sfig1}
\end{subfigure}
\hfill
\begin{subfigure} {.8\textwidth}
  \begin{center}
  \includegraphics[width=1.5\linewidth]{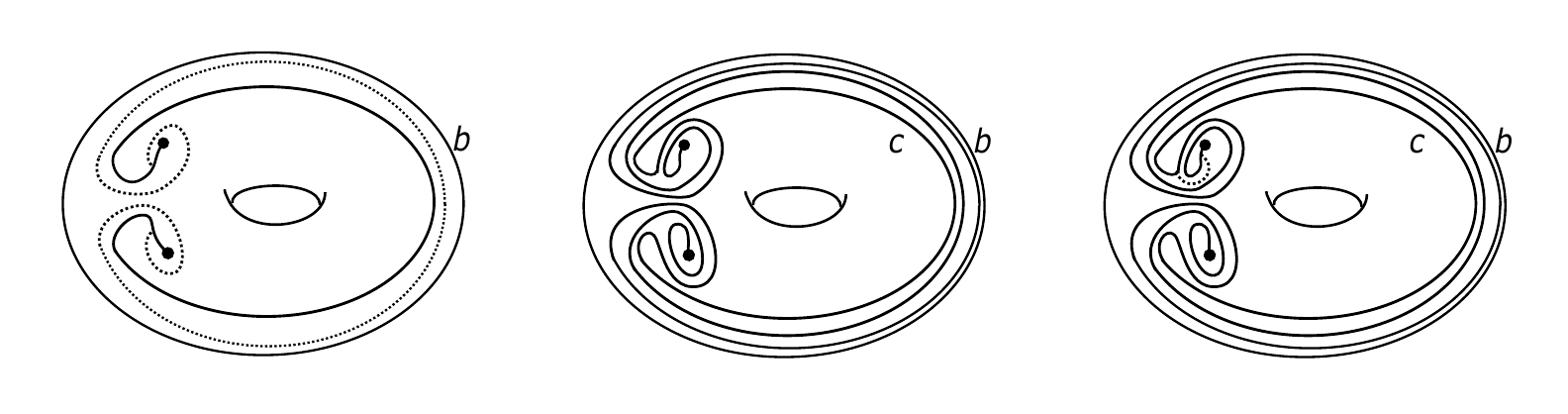}
  \end{center}
  \caption{}
  \label{casoC2:sfig2}
\end{subfigure}
\hfill
\begin{subfigure} {.8\textwidth}
  \begin{center}
  \includegraphics[height=40mm]{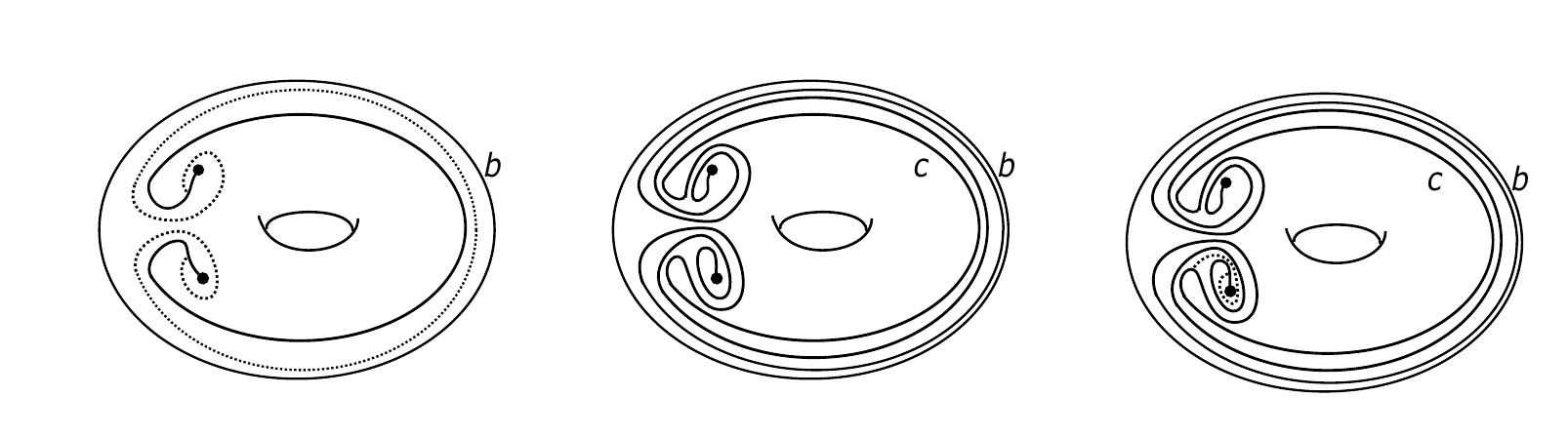}
  \end{center}
  \caption{}
\label{casoC2:sfig3}
\end{subfigure}%
\caption{Case C2: The doted lines represent the saddles.}
\label{casoC2}
\end{figure}

The second saddle changes the arc $\alpha_1$ into and arc $\alpha_2$ and a curve $c$.
The  third saddle must join $\alpha_2$ with $b$ (or $c$) into one arc $\alpha_3$, and finally this arc is joined with $c$ (or $b$) by the fourth saddle, into an arc $\alpha_4$. 

In the case that $c$ is a trivial curve, the disk bounded by $c$ must contain the arc $\alpha_2$. There are two possibilities for the third saddle which are illustrated in Figure \ref{casoC2}, in both cases we can interchange second and third saddle. In the case shown in Figure \ref{casoC2:sfig1}, we obtain a compressing disk, which contradicts the incompressibility of $F$. The other case gives a non orientable surface, which is not possible, see Figure  \ref{casoC2:sfig3}.

Thus $c$ is essential, and then $b$ and $c$ have the same slope $(p,q)$, $\vert p \vert$ and $\vert q \vert$ can not be greater or equal to 2 at the same time, otherwise we can construct a spanning annulus for $K$.

\textbf{Case C3:} Second saddle of type 1 and third saddle of type 4.

 \begin{figure}[ht]
\numberwithin{figure}{section}
\begin{center}
\includegraphics[height=40mm]{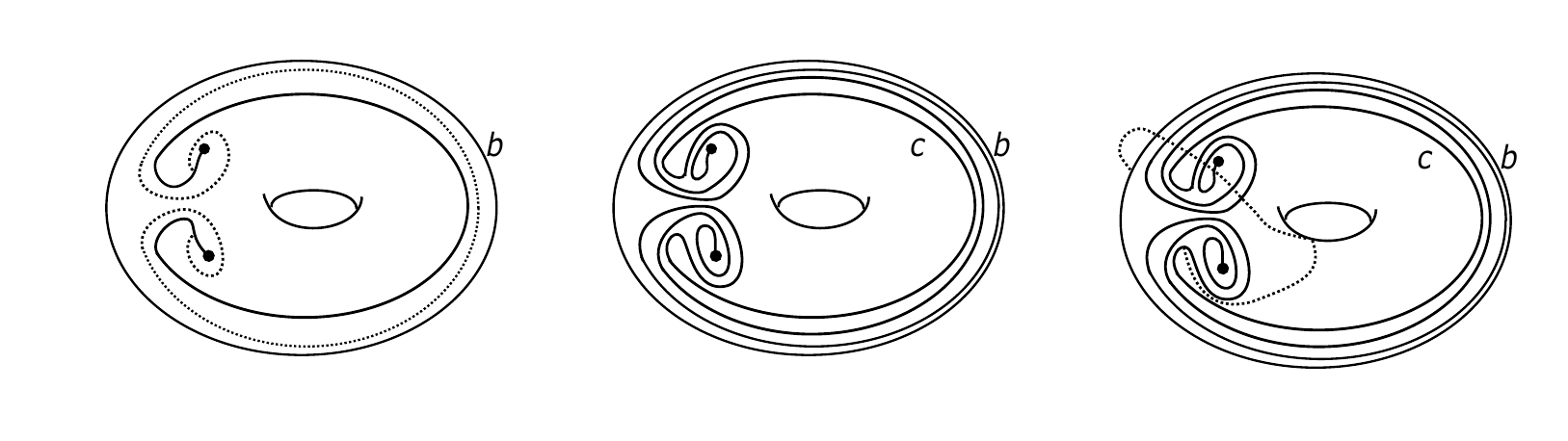}
\end{center}
\caption{Case C3}
\label{casoC3}
\end{figure}

As above, after the second saddle we have one arc $\alpha_2$ and two curves $b$ and $c$. 
If curve $c$ is inessential, the disk bounded by $c$ contains the arc $\alpha_2$, and the third saddle joins $c$ with $b$, as in Figure \ref{casoC3}. The saddles can be interchanged, in such a way that the second saddle becomes of type 2 and the third saddle of type 1, which is case \textbf{C4}.

If $c$ is essential, then third saddle must change $b$ and $c$ into a trivial curve, which is not possible by Lemma \ref{silla3tipo4}.

\textbf{Case C4:} Second saddle of type 2 and third saddle of type 1.

After the second saddle $b$ and $\alpha_1$ are transformed into an arc $\alpha_2$. After the third saddle $\alpha_2$ changes into an arc $\alpha_3$ and a curve $c$, and finally the fourth saddle changes $c$ and $\alpha_3$ into an arc $\alpha_4$. 

Notice that if $b$ and $c$ have the same slope $(p,q)$,  then $\vert p \vert$ and $\vert q \vert$ can not be greater or equal to 2 at the same time, otherwise we can construct a spanning annulus for $K$. See Figure \ref{caseC4}.

\textbf{Case C9:} Second saddle of type 3 and third saddle of type 4.

If this case happens, the second and third saddle can be level and a compressing disk is produced.

We have proved the following.

\begin{theorem}
\label{saddletheorem}
 Let $K$ be a genus two $(1,1)$-knot. Suppose $K$ is neither a torus knot nor a satellite knot. Let $F$ be a genus two Seifert surface for $K$ which satisfies $(M1)-(M4)$ with minimal complexity.  Then  component $F_0$ consists of a cancelling disk and at most one annulus; and the componente $F_1$ consists only of one cancelling disks. The sequences of types of saddles points for $\tilde{F}$ are:

\begin{center}
 
 \begin{tabular}[h]{| c | c | c | c | c |  }
 \hline
  case &1st. saddle& 2nd. saddle & 3rd. saddle &4th. saddle  \\ \hline

 B5 & type 1& type 2& type 2 & type 2 \\ 
  C2 & type 1& type 1& type 2 & type 2 \\ 
       C4 & type 1& type 2& type 1 & type 2 \\ \hline
 
 \end{tabular}

 \end{center}

\end{theorem}

\section{Classification of genus two non-satellite $(1,1)$-knots}
\label{sec4}

So far, we have understood the nature of the saddle points in the interior of  the genus two surface $F$, but still we need to give a description of the $(1,1)$-knots and the surfaces which they bound. In order to achieve such description, we perform a band sum of  two $(1,1)$-knots, which is described below.

As before, let $S$ be a standard torus in $S^3$ such that  $H_0$, $H_1$ are solid tori.

\begin{definition} (\textit{Vertical band})
A vertical band $b$ in $S\times [g,h]$, $0\leq g < h \leq 1$,  is and embedding  $b:I\times I \rightarrow S\times [g,h]$ such that for every $r \in  [g, h]$ there is a unique $t_r \in [0,1]$ such that
$b(I \times I) \cap S \times \{r\} = b(I \times \{ t_r\})$.

If $\beta_1$ and $\beta_2$ are essential simple closed curve in $S_g$ and $S_h$, respectively, and $b$ is a vertical band in $S\times [g,h]$ such that $b(I\times \{0\}) \subset \beta_1$ and $b(I\times \{1\}) \subset \beta_2$, we say that the knot $K= (\beta_1-b(I \times \{ 0\}))\cup b(\bd I \times I) \cup (\beta_2-b(I \times \{ 1\}))$ is obtained by adding a vertical band to the link $\beta_1\cup \beta_2$. Note that $K$ is a $(1,1)$-knot.
\end{definition} 

\begin{definition} (\textit{Vertical banding})
Let $K_1$ and $K_2$ be two $(1,1)$-knots such that the intersection  $\gamma_i^j= K_j \cap H_i$ is an arc for $i=0,1$ and $j=1,2$. We can embed $K_1$ in $S\times [0, \frac{1}{4}]$ and $K_2$ in $S \times [\frac{3}{4}, 1]$ in such a way that:
 \begin{itemize}
 \item $K_1 \cap S\times \{0\}= \gamma_0^1$
 \item $K_1 \cap S\times \{\frac{1}{4}\}= \gamma_1^1$
 \item $K_2 \cap S\times \{\frac{3}{4}\}= \gamma_0^2$
 \item $K_2 \cap S\times \{1\}= \gamma_1^2$
 \end{itemize} 
 
 Let $b: I \times I \rightarrow S \times [\frac{1}{4}, \frac{3}{4}]$ be a vertical band such that:
 \begin{enumerate}
 \item $b(I \times I) \cap \gamma_1^1= b(I \times \{ 0\})$,
 \item $b(I \times I) \cap \gamma_0^2= b(I \times \{ 1\})$, and
 \end{enumerate}

 \begin{figure}[ht]
\numberwithin{figure}{section}
\includegraphics[height=50mm]{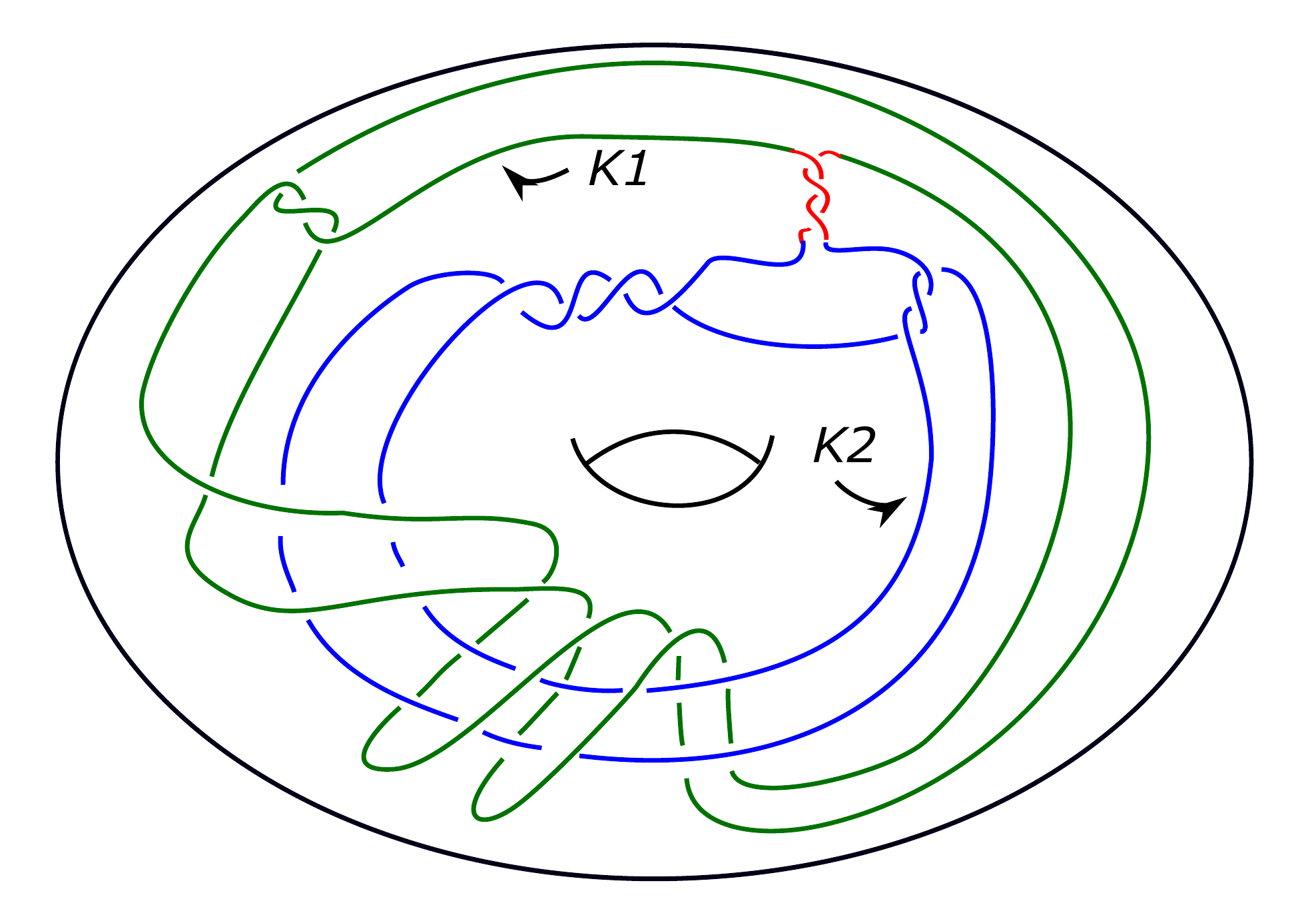}
\caption{Example of a vertical banding.}
\label{vb1}
\end{figure}

 Let $K= (K_1-b(I \times \{ 0\}))\cup b(\bd I \times I) \cup (K_2-b(I \times \{ 1\}))$. The knot $K$ will be called \textit{ vertical banding of $K_1$ and $K_2$}, and will be denoted by $K_1 \vee_b K_2$. Clearly, the knot $K_1 \vee_b K_2$ is a $(1,1)$-knot. See Figure \ref{vb1}.

 \end{definition}
 
Now we give a description of the genus 2-surface bounded by a $(1,1)$-knot $K$. We define 5 kind of pieces, which are the building blocks for the genus 2 surfaces.
\newpage
\textit{Pieces of type $P1$.}

Let $E$ be an annulus embedded in $H_0 \cup S\times [0,1/4]$, such that $E_0= E\cap H_0$ is an essential annulus in $H_0$ with slope $(1,q_0)$, $\vert q_0 \vert \geq 2$. Also, $E\cap S\times [0,1/4]$ consist of two essential vertical annuli $E_1$ and $E_2$. Let $E_3$ be a vertical essential annulus in $S\times [0,1/4]$, disjoint from $E$, such that $E_3$ lies in the parallelism region between $E$ and an annulus in $S_{1/4}$. Now let $b_1$ a vertical band in $S\times [0,1/4]$, with $b_1(I \times \{ 0\})$ contained in $E_3 \cap S_0$, and $b_1(I \times \{ 1\})$ contained in $E_i\cap S_{1/4}$, for some $i\in \{1,2,3\}$. For $j,k\not= i$, let $E_j'$ and $E_k'$ be extensions of $E_j$ and $E_k$ in $S\times [1/4,1/2]$. Let $b_2$ be a vertical band in $S\times [1/4,1/2]$, with $b_2(I \times \{ 0\})$ contained in $E_i \cap S_{1/4}$, and such that $b_1$ and $b_2$ are disjoint, and that $b_2$ is disjoint from $E_k'$ and $E_j'$. 
Let $G$ be the union of all these annuli and bands. Assume further that $G$ is orientable. If $i=3$, then the band $b_1$ must contain an even number of twists, for otherwise it would be non-orientable. Assume also that the band contains at least a twist, for otherwise the surface will be compressible. If $i=3$, it follows that $G$ is disconnected, it consists of a once punctured torus and an annulus. If $i=1,2$, it follows that $G$ is a pair of pants. Anyway, $G\cap S_{1/2}$ consists of two simple closed curves $\gamma_1$, $\gamma_2$ and an arc $\alpha_2$. If $i=3$, we say that $G$ is a piece of type $P1.1$. If $i=1,2$, and the band $b_1$ has an even number of twists (odd number of twists), we say that $G$ is a piece of type $P1.2o$ (resp. $P1.2n$). Note also that the bands could have many twists, producing many different surfaces for the same given annuli.
 
Let $A$ be an annulus embedded in $S_{1/2}$, such that $\partial A=\gamma_1\cup \gamma_2$, and such that $A$ is disjoint from $\alpha_2$. Let $\tilde G= G\cup A$. If $G$ is a piece of type $P1.1$ then $\tilde G$  has two components, a once punctured torus and a torus. If $G$ is a piece of type $P1.2o$ then $\tilde G$  is a once punctured torus, and if $G$ is a piece of type $P1.2n$ then $\tilde G$  is a once punctured Klein bottle. Let $K_1=\partial \tilde G$. Note that $K$ is a knot in a $(1,1)$-position, and that it bounds a once punctured torus or a once punctured Klein bottle. Note that there is a vertical essential annulus in $S\times [0,1/2]$ disjoint from $K$, this is just $E_1\cup E_1'$ or $E_2\cup E_2'$. Note however, that this is not an spanning annulus, for its slope is $(1,q_0)$. By pushing down the band $b_2$, we note that the knot $K_1$ is obtained by a banding of two curves, one in $S_0$ and the other in $S_{1/4}$, both of slope $(1,q_0)$. So, $K_1$ is obtained from a $(2,2q_0)$-torus link by adding a vertical band; it follows that $K_1$ is a 2-bridge knot of type $b(2l,2q_0)$
if $G$ is of type $P1.1$ or $P1.2o$, and a 2-bridge knot of type $b(2l+1,2q_0)$ if $G$ is a piece of type $P1.2n$.

\textit{Piece of type $P2$.}

Let $E_1$, $E_2$ be disjoint essential vertical annuli embedded in $S\times [0,1/2]$ and $[1/4,1/2]$,
respectively, where the slope $E_1\cap S_0$ is $(1,q_0)$, $\vert q_0 \vert \geq 2$. Now let $b_1$ a vertical band in $S\times [0,1/4]$, with $b_1(I \times \{ 0\})$ contained in $E_1 \cap S_0$, and $b_1(I \times \{ 1\})$ contained in $E_2\cap S_{1/4}$. Let $b_2$ be a vertical band in $S\times [1/4,1/2]$, with $b_2(I \times \{ 0\})$ contained in $E_2 \cap S_{1/4}$, and such that $b_1$ and $b_2$ are disjoint, and that $b_2$ is disjoint from $E_1$. Let $G$ be the union of all these annuli and bands.  Note that $G$ is a pair of pants, and that $G\cap S_{1/2}$ consists of two simple closed curves and an arc, which we denote by $\gamma_1$, $\gamma_2$ and $\alpha_2$. If the band $b_1$ has an even (odd) number of twists, we say that $G$ is a piece of type $P2o$ (resp. $P2n$). Note that there are different kinds of pieces of type $P2$, depending if the bands $b_1$ and $b_2$ lie in the same region between $E_1$ and $E_2$, or not. 

Let $A$ be an annulus embedded in $S_{1/2}$, such that $\partial A=\gamma_1\cup \gamma_2$, and such that $A$ is disjoint from $\alpha_2$. Let $\tilde G= G\cup A$. If $G$ is a piece of type $P2o$ then $\tilde G$  is a once punctured torus and a torus. If $G$ is a piece of type $P2n$ then $\tilde G$ is a once punctured Klein bottle. Let $K_1=\partial \tilde G$. Note that $K$ is a knot in a $(1,1)$-position, and that it bounds a once punctured torus or a once punctured Klein bottle. Note that there is a vertical essential annulus in $S\times [0,1/2]$ disjoint from $K$, this is just a parallel copy of $E_1$. Note however, that this is not an spanning annulus, for its slope is $(1,q_0)$.  By pushing down the band $b_2$, we note that the knot $K_1$ is obtained by a banding of two curves, one in $S_0$ and the other in $S_{1/2}$, both of slope $(1,q_0)$. So, $K_1$ is obtained from a $(2,2q_0)$-torus link by adding a vertical band; if follows that $K_1$ is a 2-bridge knot of type $b(2l,2q_0)$
if $G$ is of type $P2o$, and a 2-bridge knot of type $b(2l+1,2q_0)$ if $G$ is a piece of type $P2n$.

\textit{Piece of type $P2'$.}

It is the same as a piece of type $P2$, but inverted and contained in $S\times [1/2,1]$. So, $G\cap S_{1}$ consists of a simple closed curve of slope $(p_1,1)$, $\vert p_1 \vert \geq 2$. Also,
 $G\cap S_{1/2}$ consists of two simple closed curves and an arc. As in the previous case, we say that $G$ is a piece of type $P2o'$ or $2n'$.

\textit{Piece of type $P3$.}

Let $E_1$ be an essential vertical annulus embedded in $S\times [0,1/4]$, where the slope $E_1\cap S_0$ is $(p_0,q_0)$, $\vert q_0 \vert \geq 1$. Now let $b_1$ be a vertical band in $S\times [0,1/4]$, with $b_1(I \times \{ 0\})$ contained in $E_1 \cap S_0$, and $b_1(I \times \{ 1\})$ contained in $E_1\cap S_{1/4}$. Let $b_2$ be a vertical band in $S\times [1/4,1/2]$, with $b_2(I \times \{ 0\})$ contained in $E_1\cap S_{1/4}$, and such that $b_1$ and $b_2$ are disjoint. Let $G$ be the union of the annulus and the bands.  Assume further that $G$ is orientable; then the band $b_1$ must have an even number of twists. Note that $G$ is a once punctured torus, and that $G\cap S_{1/2}$ consists of an arc.
We say that $G$ is a piece of type $P3$. Let $\beta=b_1(\{1/2\}\times I)$ be the core of the band $b_1$. This is a monotonous arc in $S\times [0,1/4]$ intersecting $E_1$ in its endpoints. There are two possibilities for the arc $\beta$, either it is isotopic in $S\times [0,1/4]$ to an arc lying in $E_1$, or it is not. This produces two different kinds of pieces of type $P3$.

Let $K_1=\partial G$. Note that $K$ is a knot in a $(1,1)$-position, and that it bounds a once punctured torus. Note that if the arc $\beta$ is isotopic in $S\times [0,1/4]$ to an arc lying in $E_1$, then there is a vertical essential annulus in $S\times [0,1/2]$ disjoint from $G$ and then disjoint from $K_1$; this is just a parallel copy of $E_1$. But note also that if the arc $\beta$ is not isotopic in $S\times [0,1/2]$ to an arc lying in $E_1$, then there is a vertical essential annulus in $S\times [0,1/2]$ disjoint from $K_1$, but which intersects $G$ in a simple closed curve. In any case, $K_1$ is obtained by a $(2p_0,2q_0)$-torus link by adding a vertical band. If $\vert p_0 \vert \geq 2$, $\vert q_0 \vert \geq 2$, then $K_1$ is a genus one satellite $(1,1)$-knot. If $p_0=1$, then $K_1$ is a 2-bridge knot of type $b(2l,2q_0)$. If $q_0=1$, then $K_1$ is a 2-bridge knot of type $b(2l,2p_0)$. But if $p_0=0$, $q_0=1$, then in fact $K_1$ is a trivial knot, but it bounds a genus one surface which is incompressible in $H_0\cup S\times[0,1/2]$.

\textit{Piece of type $P3'$.}

It is the same as a pice of type $P3'$, but inverted and contained in $S\times [1/2,1]$. So, $G\cap S_{1}$ consists of a simple closed curve of slope $(p_1,q_1)$, $\vert p_1 \vert \geq 1$. Also,
 $G\cap S_{1/2}$ consists of an arc. We say that $G$ is a piece of type $P3'$.

\begin{theorem}
\label{piecestheorem}
 Let $K$ be a genus two $(1,1)$-knot. Suppose $K$ is neither a torus knot nor a satellite knot. Let $F$ be a genus two Seifert surface for $K$ which satisfies $(M1)-(M4)$ with minimal complexity. Then $F$ can be isotoped and decomposed into pieces, and satisfy one the following cases:
\begin{enumerate}
\item $F$ is the union of a piece of type $P1$ and a piece of type $P2'$.
\item $F$ is the union of a piece of type $P2$ and a piece of type $P2'$. 
\item $F$ is the union of a piece of type $P3$ and a piece of type $P3'$. 
\end{enumerate}
\end{theorem}

\begin{proof}
Let $K$ be a genus two $(1,1)$-knot and $F$ a Seifert surface as in Theorem \ref{saddletheorem}. Let $S_r $ be a regular level above the second saddle and below the third saddle. We can suppose, without loss of generality, that $r=1/2$. For a small $\epsilon$, the regular levels $S_{r-\epsilon}$ and $S_{r+\epsilon}$, lie below and above $S_r$, respectively. 

Following Theorem \ref{saddletheorem}, suppose first that $F$ satisfy case $B5$. Let $b$, $c$, $d$, $\alpha_1$ and $\alpha_2$ be the simple closed curves and arcs that appear in the proof of case B5 of Theorem \ref{saddletheorem}. There are two possibilities for this case.  First, on $S_r$ the curves $b$ and $c$ bound an annulus $A$ such that $\alpha_2 \notin A$, this arises when the second saddle joins the arc $\alpha_1$ with the curve $d$.  

On the levels $S_{r-\epsilon}$ and $S_{r+\epsilon}$ there are copies of $\alpha_2$, $b$, $c$ and $A$; for the sake of notation we just keep calling them the same. Consider the surfaces

\begin{center}
$G_1= F_0 \cup (F \cap S\times [0, r-\epsilon])$ \\
$G_2= F_1 \cup (F \cap S\times [r+\epsilon, 1])$
\end{center}

Note that after the first saddle, the canceling disk transforms into an annulus, and the second saddle adds a vertical band to this annulus. So $G_1$ is a disconnected surface, it consists of a once punctured torus and an annulus. By reading backwards the saddle points, after the fourth saddle, the canceling disk in $F_1$ is transformed into an annulus, and after the third saddle, this annulus is transformed into a pair of pants, so $G_2$ is a pair of pants. It follows that $G_1$ is a piece of type $P1.1$ and $G_2$ is a piece of type $P2o'$. It cannot be a piece of type $P2n'$, for if this happens, the surface $F$ would not be orientable. Consider the following knots: 

\begin{center}
$K_1= \alpha_2 \cup (K\cap S\times [0, r-\epsilon])$\\

$K_2= \alpha_2 \cup (K\cap S\times [r+\epsilon, 1])$
\end{center}

Both are  $(1,1)$-knots. For a schematic picture, see Figure \ref{vbB5-1}. By the description of the pieces of type $P1$ and $P2'$, it follows that $K= K_1 \vee_b K_2$, where $K_1$ and $K_2$ are genus one 2-bridge knot. Namely, $K_1$ is a 2-bridge knot of type $b(2l,2q_0)$, and $K_2$ is a 2-bridge knot of type $b(2m,2q_0)$, for certain integers $l,m$.

The other possibility for case $B_5$, is that the second saddle joins the arc $\alpha_1$ with the curve $b$.   On $S_r$ the curves $d$ and $c$ bound an annulus $A$ such that $\alpha_2 \notin A$. Again, consider the surfaces

\begin{center}
$G_1= F_0 \cup (F \cap S\times [0, r-\epsilon])$ \\
$G_2= F_1 \cup (F \cap S\times [r+\epsilon, 1])$
\end{center}

Note that after the first saddle, the canceling disk transforms into an annulus, and the second saddle connects, via a vertical band, this annulus to the annulus coming from $F_0$. So $G_1$ is a pair of pants.
By an argument as in the previous case, $G_2$ is also a pair of pants.  It follows that $G_1$ is a piece of type $P1.2o$  ($P1.2n$) and that $G_2$ is a piece of type $P2o'$ (resp. $P2n'$). Again we can consider the following knots: 

\begin{center}
$K_1= \alpha_2 \cup (K\cap S\times [0, r-\epsilon])$\\

$K_2= \alpha_2 \cup (K\cap S\times [r+\epsilon, 1])$
\end{center}

Both are  $(1,1)$-knots, and it follows that $K= K_1 \vee_b K_2$. See Figure \ref{vbB5-2}. If $G_1$ is a piece of type $P1.2o$ and $G_2$ is a piece of type $P2o'$, then $K_1$ and $K_2$ are genus one 2-bridge knots. Namely, $K_1$ is a 2-bridge knot of type $b(2l,2q_0)$, and $K_2$ is a 2-bridge knot of type $b(2m,2q_0)$, for certain integers $l,m$. If $G_1$ is a piece of type $P1.2n$ and $G_2$ is a piece of type $P2n'$, then $K_1$ and $K_2$ are cross cap number two 2-bridge knots. Namely, $K_1$ is a 2-bridge knot of type $b(2l+1,2q_0)$, and $K_2$ is a 2-bridge knot of type $b(2m+1,2q_0)$, for certain integers $l,m$.

\begin{figure}[h]
\numberwithin{figure}{section}
\includegraphics[height=80mm]{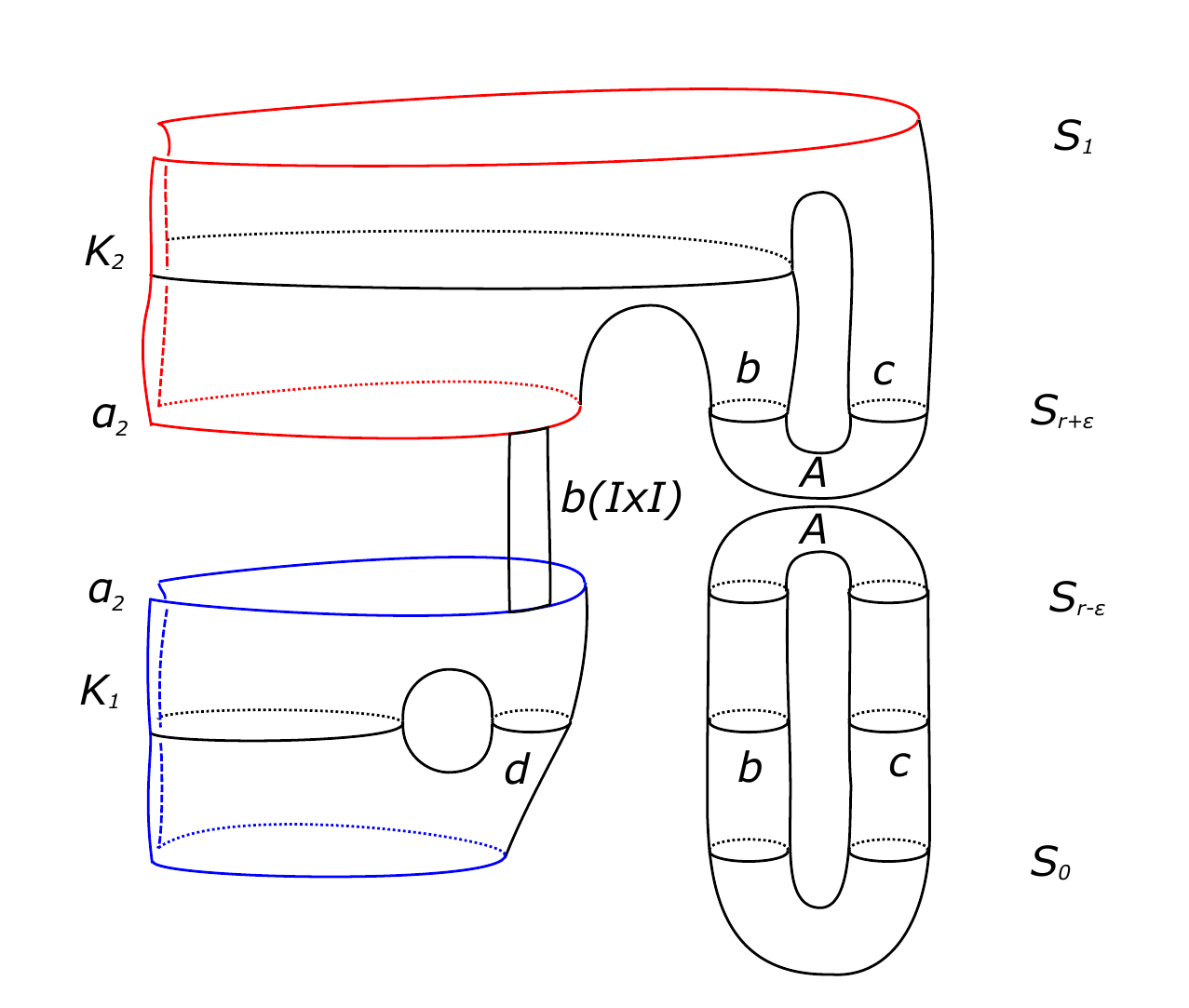}
\caption{Vertical banding for first case B5.}
\label{vbB5-1}
\end{figure}

\begin{figure}[h]
\numberwithin{figure}{section}
\includegraphics[height=80mm]{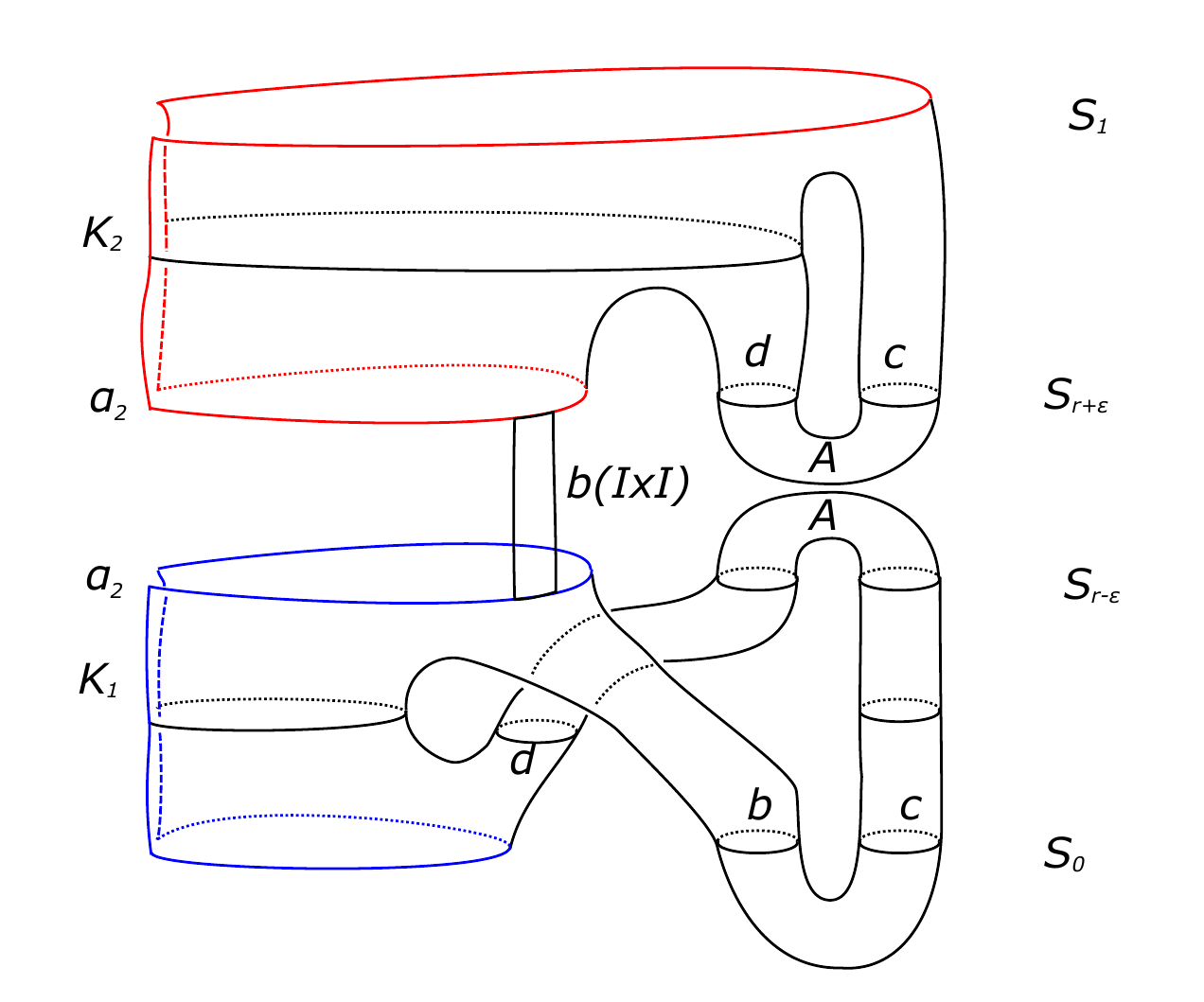}
\caption{Vertical banding for second case B5.}
\label{vbB5-2}
\end{figure}

If $F$ is as in case $C2$, then on $S_r$ the curves $b$ and $c$ bound an annulus $A$ such that $\alpha_2 \notin A$. On the levels $S_{r-\epsilon}$ and $S_{r+\epsilon}$ there are copies of $\alpha_2$, $b$, $d$ and $A$, for the sake of notation we just keep calling them the same. Consider the surfaces:

\begin{center}

$G_1= F_0 \cup (F \cap S\times [0, r-\epsilon])$ \\
$G_2= F_1 \cup (F \cap S\times [r+\epsilon, 1])$
\end{center}

After the first saddle, the canceling disk in $F_1$ is transformed into an annulus, and after the second saddle, this annulus is transformed into a pair of pants. The same argument show that $G_2$ is also a pair of pants. It follows that $G_1$ is a piece of type $P2o$ or $P2n$, and $G_2$ is a piece of type $P2o'$ or $P2n'$, respectively. 

Consider the following knots: 
\begin{center}
$K_1= \alpha_2 \cup (K\cap S\times [0, r-\epsilon])$\\

$K_2= \alpha_2 \cup (K\cap S\times  [r+\epsilon, 1])$
\end{center}

Both are  $(1,1)$-knots. It follows that $K= K_1 \vee_b K_2$. A schematic picture is shown in Figure \ref{vbC2}. If $G_1$ is a piece of type $P2o$ and $G_2$ is a piece of type $P2o'$, then $K_1$ and $K_2$ are genus one 2-bridge knot. Namely, $K_1$ is a 2-bridge knot of type $b(2l,2q_0)$, and $K_2$ is a 2-bridge knot of type $b(2m,2q_0)$, for certain integers $l,m$. If $G_1$ is a piece of type $P2n$ and $G_2$ is a piece of type $P2n'$, then $K_1$ and $K_2$ are cross cap tunnel number two 2-bridge knots. Namely, $K_1$ is a 2-bridge knot of type $b(2l+1,2q_0)$, and $K_2$ is a 2-bridge knot of type $b(2m+1,2q_0)$, for certain integers $l,m$.

\begin{figure}[h]
\numberwithin{figure}{section}
\includegraphics[height=80mm]{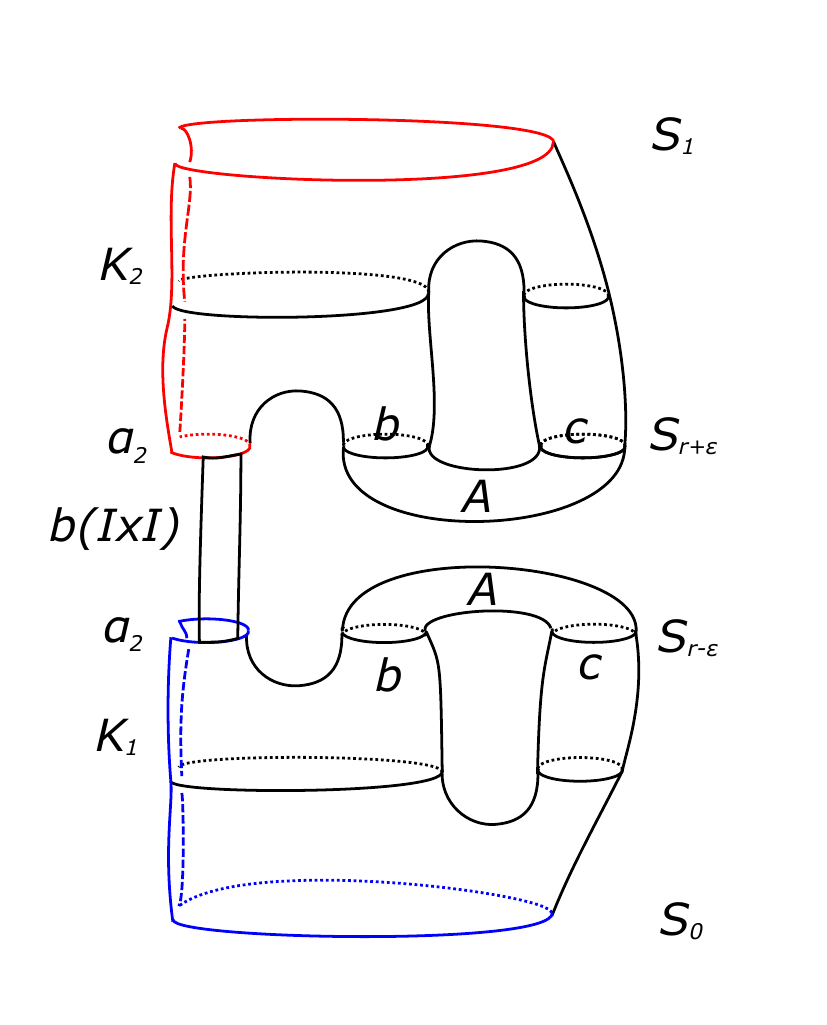}
\caption{Vertical banding for case C2.}
\label{vbC2}
\end{figure}

\begin{figure}[h]
\numberwithin{figure}{section}
\includegraphics[height=80mm]{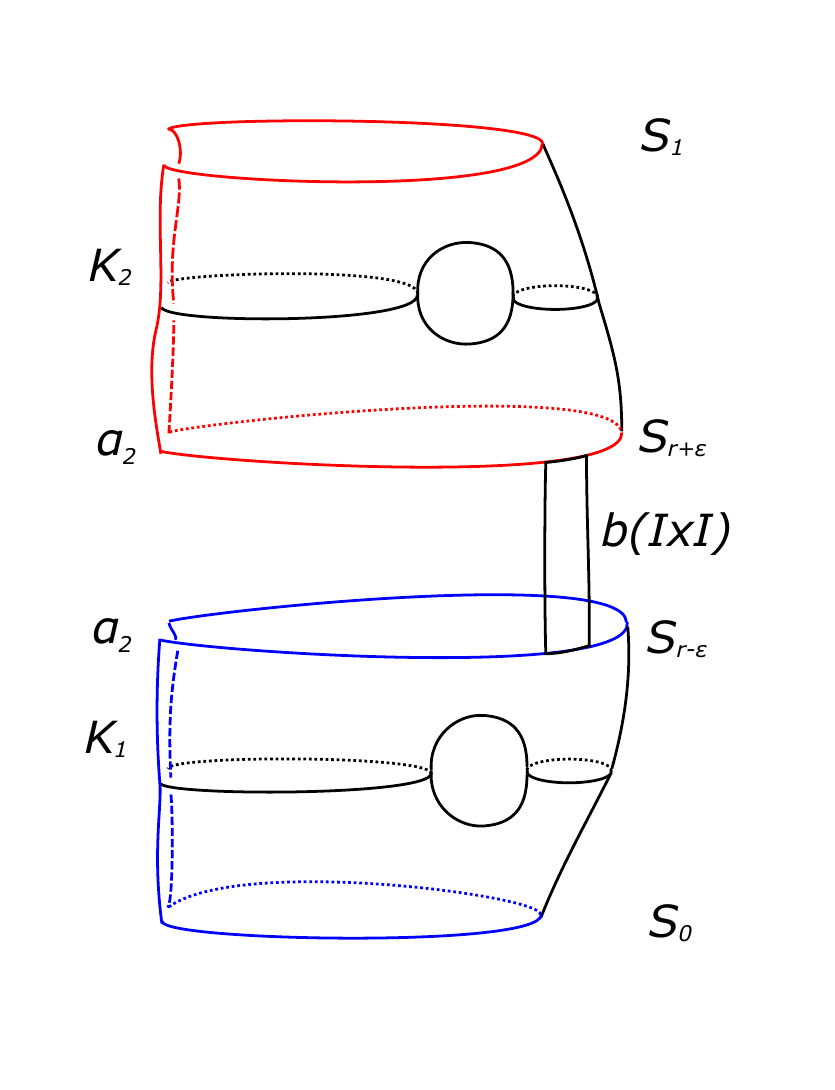}
\caption{Vertical banding for case C4.}
\label{vbC4}
\end{figure}

If $F$ satisfies case $C4$, then  $S_r \cap F= \alpha_2$.  On the levels $S_{r-\epsilon}$ and $S_{r+\epsilon}$ we see copies of $\alpha_2$, which we denote by $\alpha_2$. Consider the surfaces:

\begin{center}
$G_1= F_0 \cup (F \cap S\times [0, r-\epsilon])$ \\
$G_2= F_1 \cup (F \cap S\times [r+\epsilon, 1])$
\end{center}

Note that after the first saddle, the canceling disk transforms into an annulus, and the second saddle adds a vertical band to this annulus. So $G_1$ is a once punctured torus. The same argument, by reading backwards the saddle points, shows that $G_2$ is also a once punctured torus. Then $G_1$ is a piece of type $P3$ and $G_2$ is a piece of type $P3'$.
 
Consider the knots:
\begin{center}
$K_1= \alpha_2 \cup (K\cap S\times [0, r-\epsilon])$\\
$K_2= \alpha_2 \cup (K\cap S\times [r+\epsilon, 1])$
\end{center}

Both are  $(1,1)$-knots. $K_i$ is a $(1,1)$-knot of genus one, since it bounds $G_i$, for $i=1,2$. The knot $K$ is the vertical banding of $K_1$ and $K_2$. See Figure \ref{vbC4}.

Let $(p_1, q_1)$ be the slope of $b$ and let $(p_2, q_2)$ be the slope of $c$. If $\vert p_i \vert$, $\vert q_i \vert \geq 2$, for $i=1,2$,  then $(p_1, q_1)\neq (p_2, q_2)$, otherwise we can obtain a spanning annulus for  $K$. Therefore $K_1$ and $K_2$ can be satellites of different types. If for some $i=1,2$, say $i=1$,  $\vert p_1 \vert =1$, then $K_1$ is a genus one 2-bridge knot. For $(p_2, q_2)$ it can happen that $\vert p_2 \vert$, $\vert q_2 \vert \geq 2$, thus $K_2$ is a satellite, otherwise  $K_2$ is a genus one 2-bridge knot. 

Note that one or both of $K_1$ and $K_2$ can be trivial knots. This happens when $(p_0,q_0)=(0,1)$, or $(p_1,q_1)=(1,0)$.
\end{proof}

The previous theorem has the following consequence.

\begin{theorem}
\label{nonsatcase}
Let $K$ be a genus two $(1,1)$-knot satisfying the conditions of Theorem \ref{piecestheorem}. Then $K$ is a vertical banding, $K_1 \vee_b K$, where $K_1$ and $K_2$ are one the following cases:
\begin{itemize}
\item Both $K_1$ and $K_2$ are 2-bridge $(1,1)$-knots of  genus one; or
\item Both $K_1$ and $K_2$ are satellite $(1,1)$-knots of  genus one; or
\item $K_1$ is a satellite $(1,1)$-knot of  genus one and $K_2$ is a 2-bridge $(1,1)$-knot of  genus one; or
\item Both $K_1$ an $K_2$ are cross cap number two 2-bridge knots; or
\item $K_1$ is a trivial knot in a certain position, and $K_2$ is a genus one 2-bridge knot, or a satellite genus one $(1,1)$-knot, or a trivial knot in certain position.
\end{itemize}

\end{theorem}

\begin{ejem}
Let us consider a vertical banding $K$ of two 2-bridge knots as shown in Figure \ref{2bvb:sfig1}, together with a Seifert surfaces satisfying $(M1)-(M4)$. We construct an example satisfying case (3) of Theorem \ref{piecestheorem}. Take a piece of Type $P3$ in $S\times[0,1/2]$, consisting of a vertical annulus of slope $(q,1)$, and a vertical band with $2l$ twists, plus a band with no twists. Take also a piece of type $P3'$ in $S\times[1/2,1]$, consisting of a vertical annulus of slope $(p,1)$, and a vertical band with $2k$ twists, plus a band with no twists. Let us call such a vertical banding an untwisted vertical banding. In Figure  \ref{2bvb:sfig1} we are showing the case $q=2$, $l=1$, $p=2$, $k=2$. Isotope $K$, by sliding the Seifert surface  as in Figures \ref{2bvb:sfig2}, \ref{2bvb:sfig3}, \ref{2bvb:sfig4},  the knot $K$ looks like the one in Figure \ref{2bvb:sfig5}. Observe that in the general case, $K$ is isotopic to the top of Figure \ref{2bvb8}. It is not difficult to see that $K$ is isotopic to the knot shown in the bottom  of Figure \ref{2bvb8}, which clearly it is a 2-bridge knot. It is well known that any 2-bridge knot of genus 2 can be expressed as in the bottom of Figure \ref{2bvb8}. This shows that any 2-bridge knot of genus two can be expressed as in case (3) of Theorem \ref{piecestheorem}.
\end{ejem}

\begin{figure}
\begin{subfigure}{.5\textwidth}
  \centering
  \includegraphics[width=.8\linewidth]{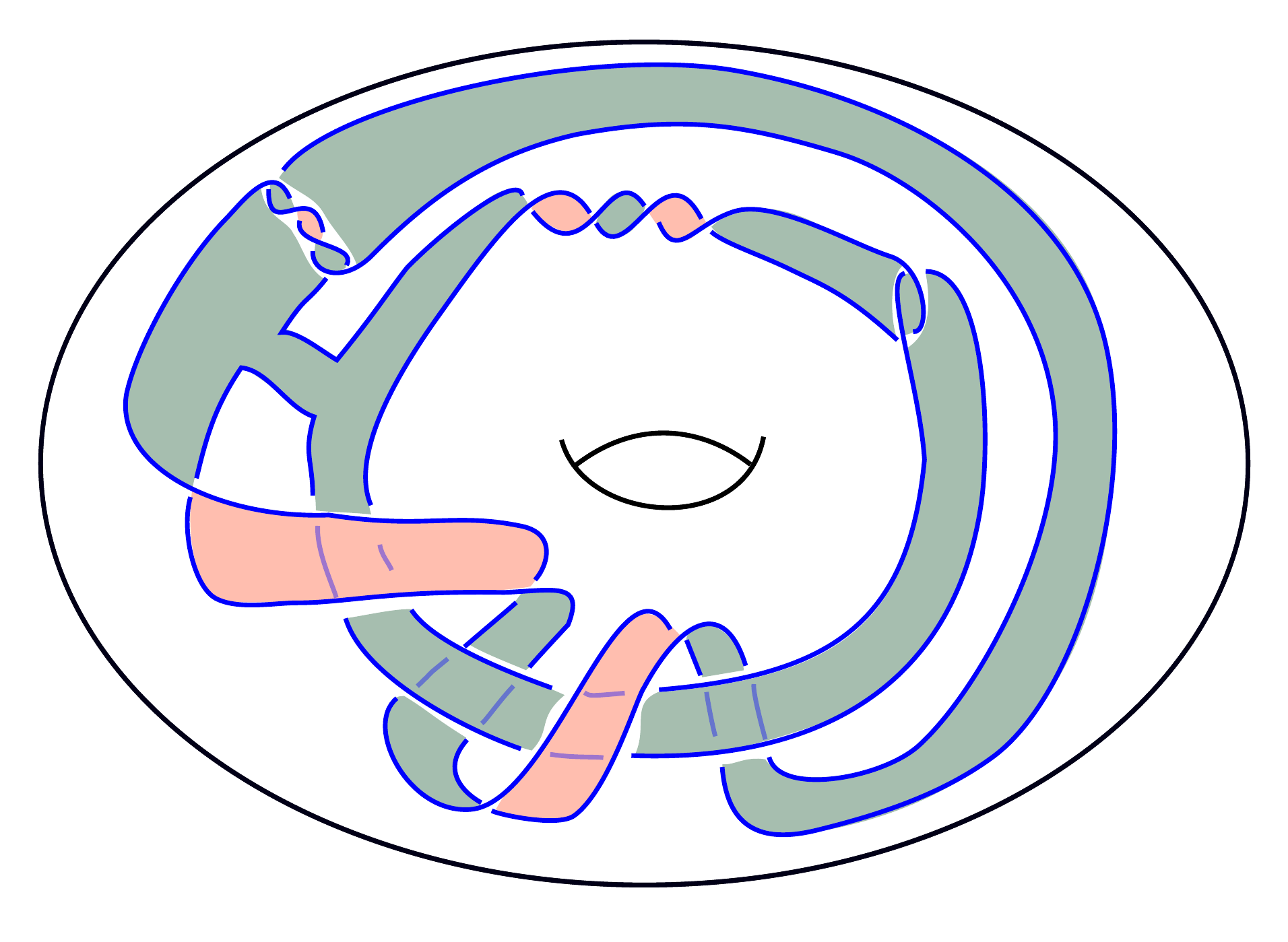}
  \caption{}
  \label{2bvb:sfig1}
\end{subfigure}%
\begin{subfigure}{.5\textwidth}
  \centering
  \includegraphics[width=.8\linewidth]{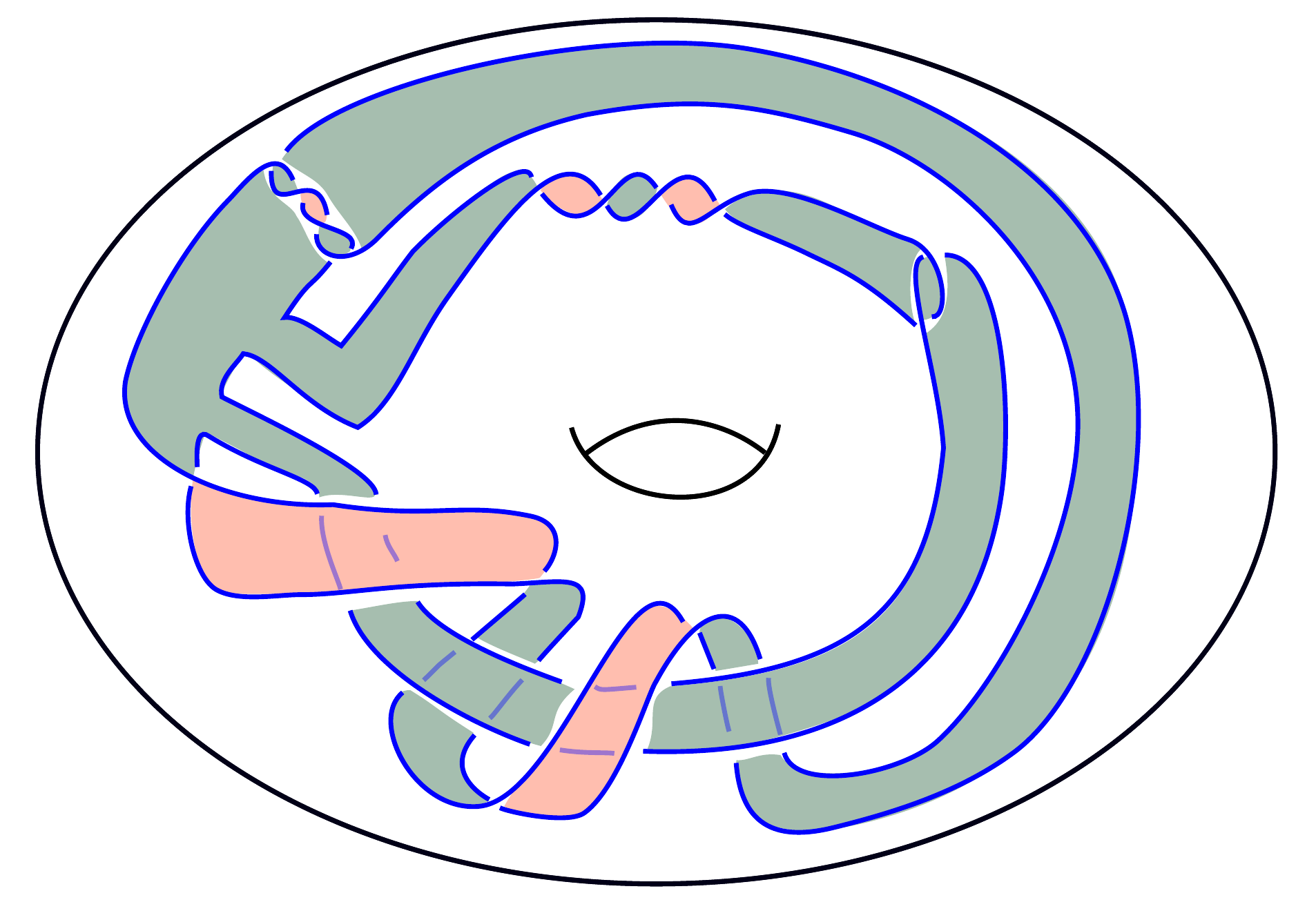}
  \caption{}
  \label{2bvb:sfig2}
\end{subfigure}
\begin{subfigure}{.5\textwidth}
  \centering
  \includegraphics[width=.8\linewidth]{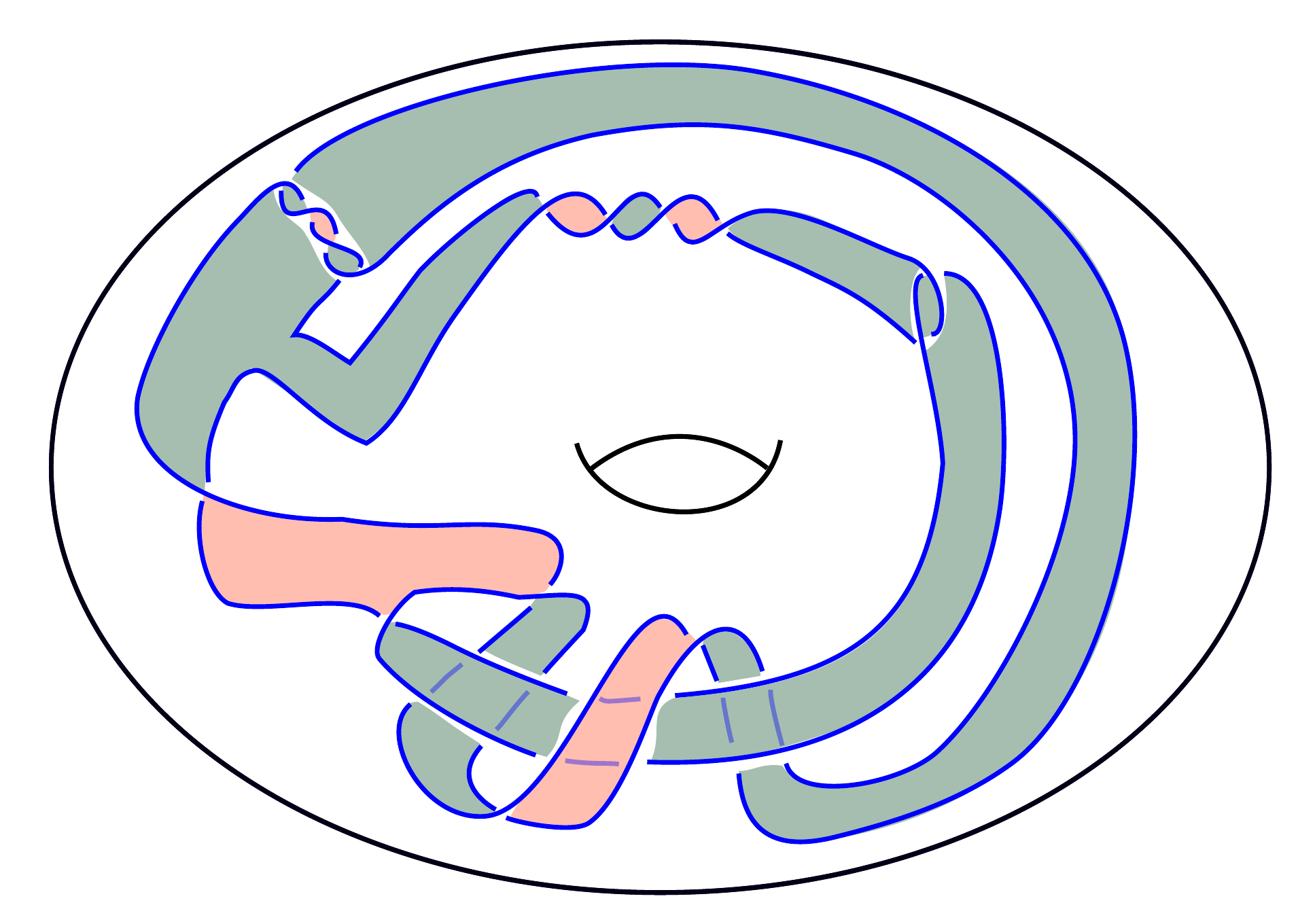}
  \caption{}
\label{2bvb:sfig3}
\end{subfigure}%
\begin{subfigure}{.5\textwidth}
  \centering
  \includegraphics[width=.8\linewidth]{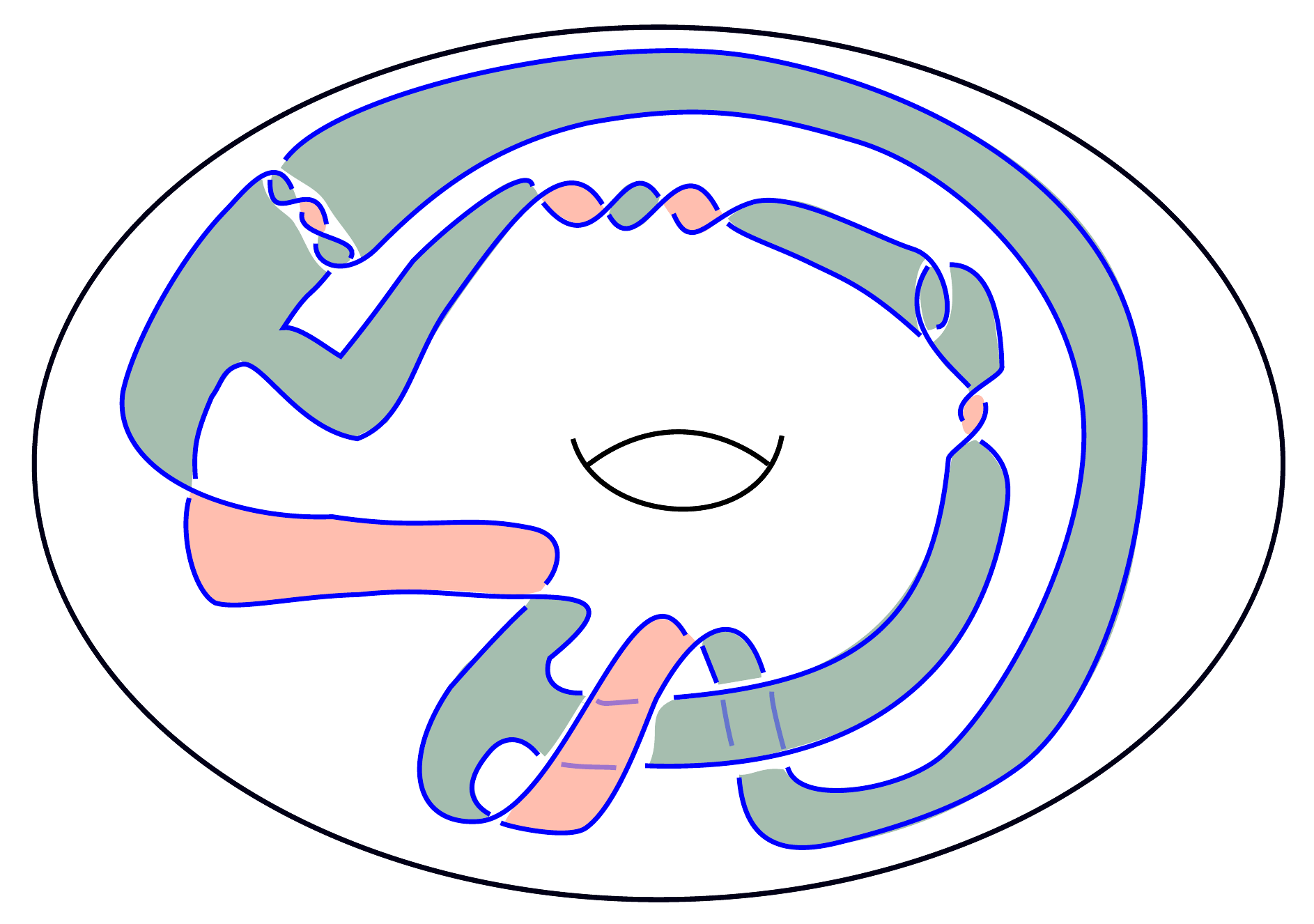}
 \caption{}
 \label{2bvb:sfig4}
\end{subfigure}
\begin{subfigure}{.5\textwidth}
  \centering
  \includegraphics[width=.8\linewidth]{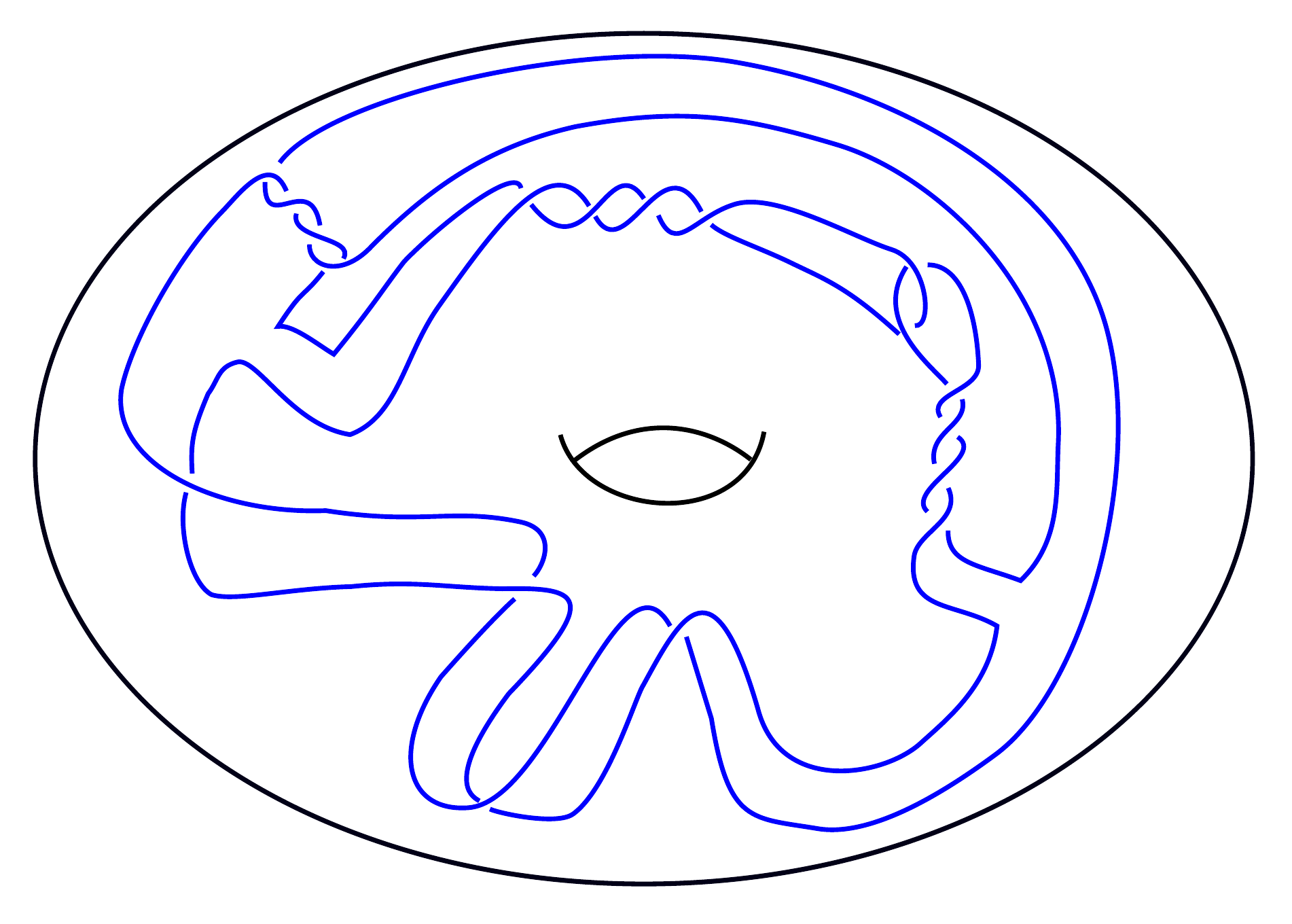}
 \caption{}
 \label{2bvb:sfig5}
\end{subfigure}%
\begin{subfigure}{.5\textwidth}
  \centering
  \includegraphics[width=.8\linewidth]{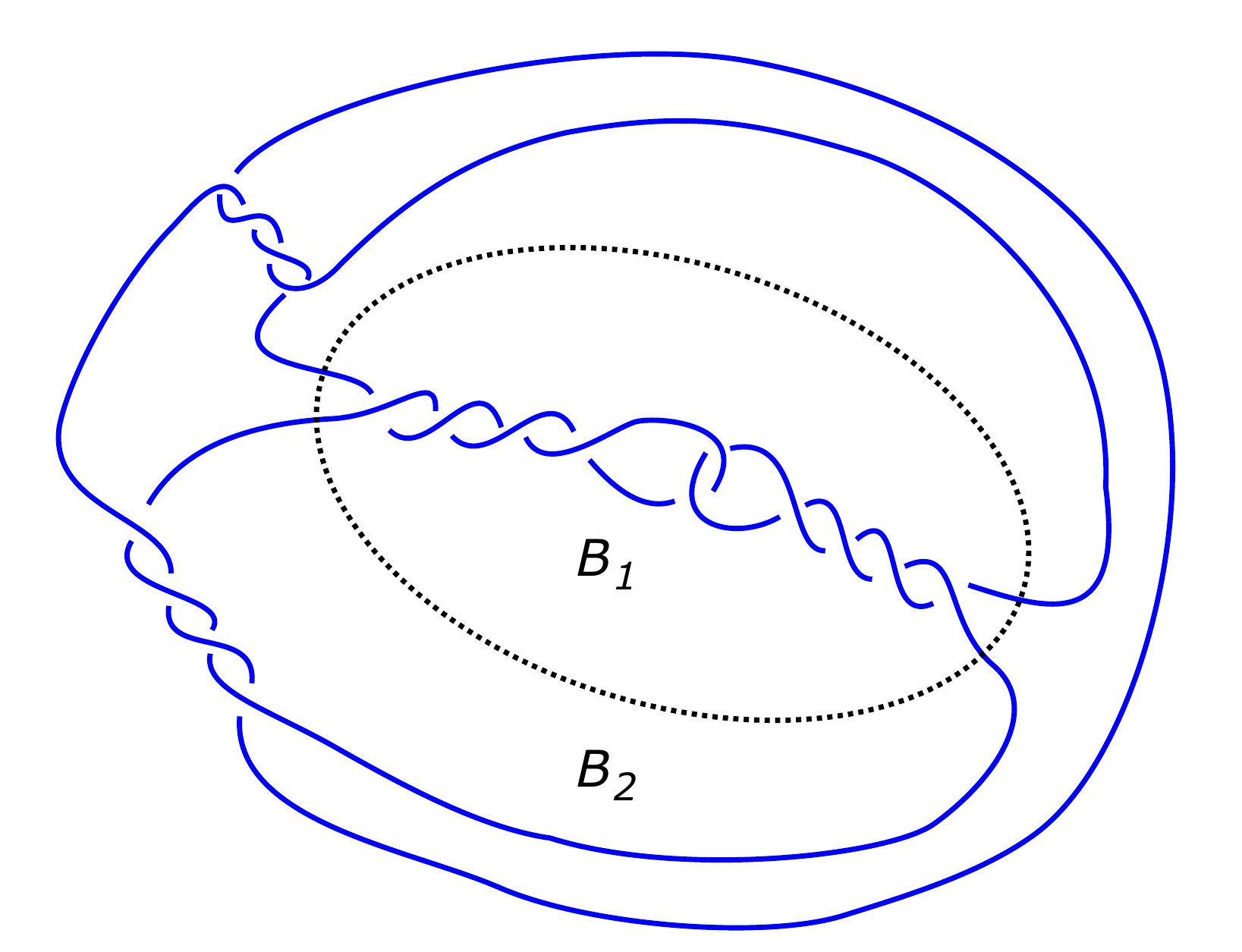}
 \caption{}
\label{2bvb:sfig6}
\end{subfigure}
\caption{Isotopy of a untwisted vertical banding of two 2-bridge knots.}
\label{2bvb}
\end{figure}

\begin{figure}[h]
\numberwithin{figure}{section}
\includegraphics[height=70mm]{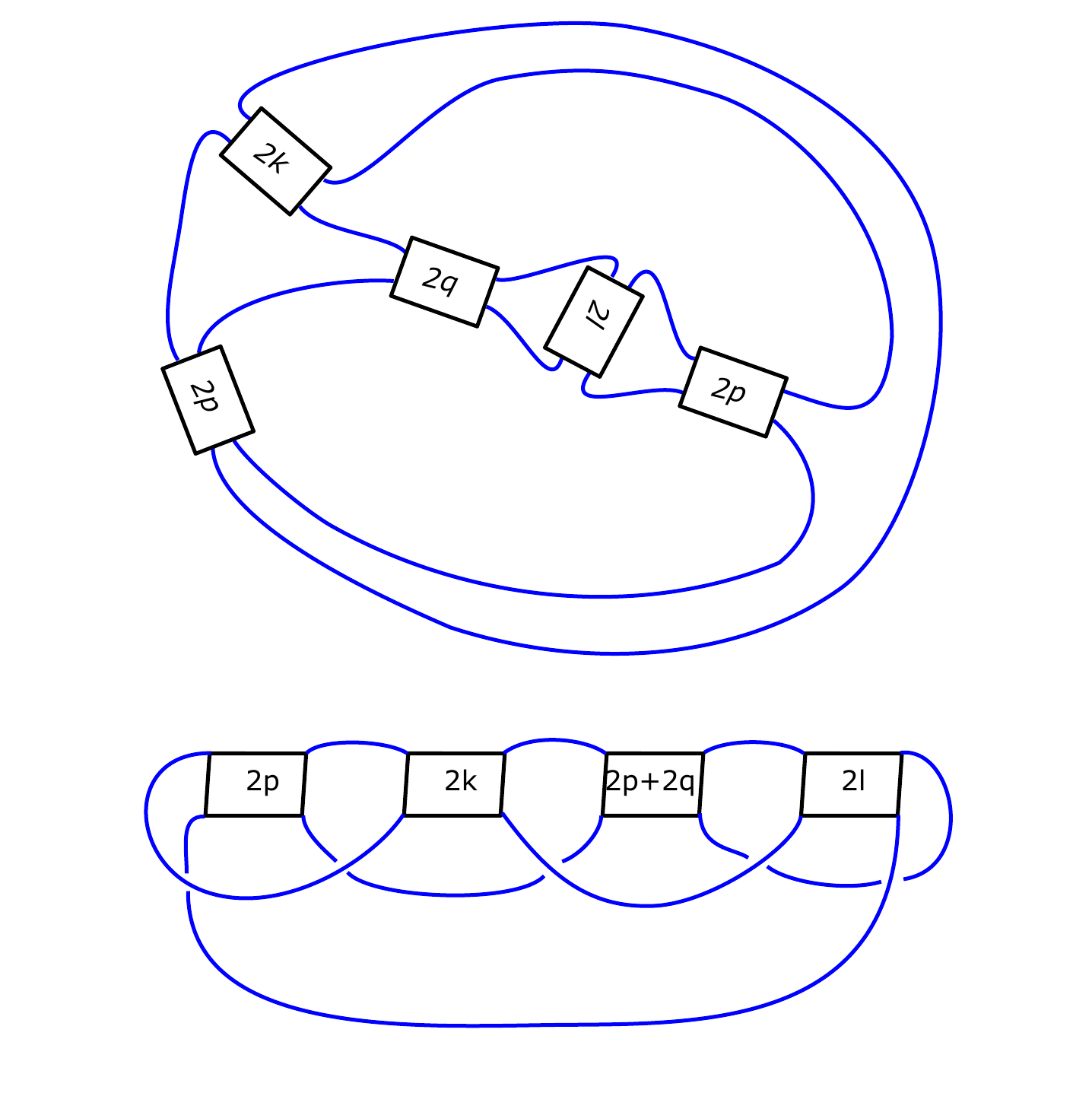}
\caption{Tangle sum of two rational tangles isotopic to a vertical banding of  two 2-bridge knots.}
\label{2bvb8}
\end{figure}

\begin{ejem}
Now we consider a vertical banding $K$ of two satellite $(1,1)$-knots of genus one. So, we construct examples satisfying case (3) of Theorem \ref{piecestheorem}. Take a piece of Type $P3$ in $S\times[0,1/2]$, consisting of a vertical annulus $E_1$ of slope $(p_0,q_0)$, where $\vert p_0 \vert$, $\vert q_0 \vert \geq 2$, a vertical band $b_1$ with $2k$ twists such that the core of the band is isotopic to an arc on the annulus $E_1$, plus a vertical band $b_2$. Take also a piece of type $P3'$ in $S\times[1/2,1]$, consisting of a vertical annulus $E_2$ of slope $(p_1,q_1)$, where $\vert p_1 \vert$, $\vert q_1 \vert \geq 2$, a vertical band $b_1'$ with $2l$ twists such that the core of the band is isotopic to an arc on the annulus $E_2$, plus a vertical band $b_2'$. Suppose also that $\vert p_0q_1-q_0p_1 \vert \geq 2$.  Let $F$ be the union of the annuli and bands, so $F$ is a genus two surface. Let $N(E_1)$ and $N(E_2)$ be regular neighborhoods of $E_1$ and $E_2$ which contain the bands $b_1$ and $b_1'$ respectively. Let $H=N(E_1)\cup N(b_e\cup b_2') \cup N(E_2)$; note that $H$ is a genus two handlebody such that $F\subset H$. It follows from the assumptions on the slopes and the main result of \cite{EuM}, that $\partial H$ is a closed genus two surface, which is incompressible in $E(K)$. This shows that there exists genus two $(1,1)$-knots which contain a closed incompressible of genus two in its exterior. 

\end{ejem}


\section{Satellite genus two knots}
\label{sec5}

Morimoto and Sakuma \cite{MS} determined the knot types of satellite tunnel number one knots in $S^3$. These knots are constructed as follows. Let $K_0$ be a  $(p,q)$-torus knot in $S^3$ with $p\neq 1$ and $q\neq 1$, and let $L= K_1 \cup K_2$ be a 2-bridge link of type $(\alpha, \beta)$ in $S^3$ with $\alpha \geq 4$. Note that $K_0$ is a non-trivial knot, and $L$ is neither a trivial link nor a Hopf link. Since $K_2$ is the trivial knot in $S^3$, there is a an orientation preserving homeomorphism $f: E(K_2) \rightarrow N(K_0)$ which takes a meridian $m_2 \subset \bd E(K_2)$ of $K_2$ to a fiber $h \subset \bd N(K_0)=\bd E(K_0)$ of the unique Seifert fibration of $E(K_0)$. The knot $f(K_1)\subset N(K_0) \subset S^3$ is denoted by the symbol $K(\alpha, \beta; p, q)$. Every satellite knot of tunnel number one has the form $K(\alpha, \beta; p, q)$ for some integers $\alpha, \beta, p, q$. Eudave-Mu\~noz \cite{Eu} obtained another description of these knots. These knots are known to admit $(1,1)$-decompositions.

The aim of this section if to prove the following:

\begin{theorem}
\label{teoclasificacion}
Let $K$ be a satellite tunnel number one  genus two knot in $S^3$. Then $K=K(\alpha, \beta; p, q)$, where $\frac{\alpha}{\beta}$ is given by one of the continued fractions:
\begin{enumerate}
\item  $[2, 2u, 2, 2v+1, 2, 2w,2]$ 
\item $[2, 2u+1, 2, 2v, 2, 2w+1,2]$ 
\end{enumerate}
with $u,v,w \in \mathbb{Z} -\{0\}$
\end{theorem}

Let $l$ and $m$ be a preferred longitude and a meridian for $\bd N(K_0)$, respectively. 
Notice that $\Delta(l, h)=pq$ and then $\Delta(f^{-1}(l), m_2)=pq$.

The next lemma can be found in \cite{BZ}, and will be useful.

\begin{lemma}
Let $K=K(\alpha, \beta; p, q)$ be a satellite tunnel number one knots. Let $F$ be a minimal genus Seifert surface for $K$. The surface $F$ can be isotoped in such a way that $F \cap \bd N(K_0)$ consists of $r$ preferred longitudes and $F \cap (S^3 -N(K_0))$ is made of $r$ components which are Seifert surfaces for $K_0$.
\end{lemma}

Suppose the 2-bridge presentation of $L$ is given relative to some 2-sphere $S$ in $S^3$, notice that in this case $H_0$ and $H_1$ are 3-balls. 
The following is Lemma 3.2 of \cite{RV}.

\begin{lemma}
\label{KT}
Let $F$ be a surface in $S^3$ spanned by $K_1$ (orientable or not) and transverse to $K_2$, such that $F'= F \cap E(L)$ is essential and meridionally incompressible in $E(L)$. If $F'$ is isotoped so as to satisfy (M1)-M(4) with minimal complexity, then $\vert Y(F')\vert= 2- (\chi(F')+ \vert \bd F' \vert) $, and
\begin{enumerate}
\item each critical point of $h\vert \tilde{F}$ is a saddle,
\item for $0\leq r \leq 1$ any circle of $S_r \cap F$ is  nontrivial in $S_r- L$ and $F$, and
\item $F_0$ and $F_1$ each consists of one cancelling disk.
\end{enumerate}
\end{lemma}

\begin{lemma}
Let $K=K(\alpha, \beta; p, q)$ be a satellite tunnel number one knot. Assume $K$ is of genus two and is not a torus knot. Let $F$ be a minimal Seifert surface for $K$, then $F \subset N(K_0)$.
\end{lemma}
\begin{proof}
Let us assume that $F$ has been isotoped so that $F \cap \bd N(K_0)\neq \emptyset$ consists of preferred longitudes. The intersection $F \cap N(K_0)$ is either an annulus, or a pair of pants, or a torus with two holes. 

If $F \cap  N(K_0)$ is an annulus, then $K$ is a torus knot, which is not the case.

Now, suppose $F \cap  N(K_0)$ is a pair of pants $P$ whose  boundary components consist of  $K$ and two preferred longitudes $\lambda_1$ and $\lambda_2$. Since $F \cap (S^3-N(K_0))$ is a genus one surface for  $K_0$, then $K_0$ is a  $(3,2)$-torus knot. For $m$, a meridian of $K_0$, $\Delta (\lambda_i, m)=1$, then $\Delta(f^{-1}(\lambda_i), l_2)=1$ with $l_2$ a preferred longitude of $K_2$. For $l$, a preferred longitude of $K_0$, $\Delta (\lambda_i, l)=6$, then $\Delta(f^{-1}(\lambda_i), m_2)=6$ with $m_2$ a meridian of $K_2$. Then $f^{-1}(P)$ is a pair of pants in $E(K_2 \cup K_1)$ whose boundary components consist of one preferred longitude in $\bd N(K_1)$ and two curves of slope $1/6$ in $\bd N(K_2)$.

As $K_2$ is a trivial knot in $S^3$, by performing $1/6$ Dehn surgery on $K_2$ we obtain $S^3$ again and then $K_1$ becomes a knot $K_1'$. The surface $f^{-1}(P)$ after surgery is then a surface with one boundary component a longitude of $K_1'$ and the other two are meridians of $K_2'$, the core of the surgered solid torus. Then $K_1'$ is a trivial knot for it bounds a disk, namely the image of $f^{-1}(P)$ union two meridians disks of $K_2'$. As $K_1$ and $K_1'$ are both trivial knots, this not possible by the main result of \cite{M}, since $K_1 \cup K_2$ is not the trivial link neither the Hopf link.

Finally if $F \cap  N(K_0)$ is torus $T$ with two boundary components consisting of $K$, and $\lambda_1$ a preferred longitude for $\bd N(K_0)$. By a similar argument as above, $f^{-1}(T) \subset E(K_2)$ is a torus in $E(K_2 \cup K_1)$  whose boundary components consist of one preferred longitude in $\bd N(K_1)$, and one curve of slope $1/6$ in $\bd N(K_2)$.

Suppose a 2-bridge presentation of $L=K_1\cup K_2$ is given relative to some 2-sphere $S$ in $S^3$, notice that in this case $H_0$ and $H_1$ are 3-balls. 

The surface $T'=f^{-1}(T)$ can be isotoped  to satisfy $M1- M4$ with minimal complexity so that $T'_i$ consists of seven cancelling disks, one in $K_1$ and the others in $K_2$, for $i=0,1$. This can be achieved by arguments similar to Lemma  \ref{KT}. 
By Lemma 7.1 of \cite{FH}, $\vert Y(T')\vert \geq 6$; on the other hand, Lemma  \ref{KT} implies that $\vert Y(T')\vert=2$. This is a contradiction.

Therefore,  $F \subset N(K_0)$.

\end{proof}

The lemma above implies that $F'=f^{-1}(F) \subset E(K_2)$. Moreover $F'$ is an incompressible genus two Seifert surface for $K_1$.

\begin{lemma}
\label{meri-comp}
 The surface $F'$ is meridionally compressible. Moreover, $F'$ can be meridionally compressed twice to obtain a  disk $\Sigma$ that satisfies the conditions of Lemma \ref{KT}.
\end{lemma}

\begin{proof}
Suppose $F'$ is meridionally incompressible, Lemma \ref{KT} can be used. Let $\alpha_0$ the arc in $F' \cap S_0$. The first saddle must change $\alpha_0$ into an arc and a curve $\gamma$. Let $S_r$ be the level where the curve $\gamma$ lies, then $\gamma$ either separates or does not separate the points $S_r \cap K_2$, the first option is not possible by Lemma \ref{KT} $(2)$ since $F'$ is meriodionally incompressible, while the second option using the cancelling disk $F'_0$, $F'$ compresses into $E(L)$ along $\gamma$. Therefore, $F'$ is meridionally compressible.

Let $\gamma$ be a circle along which $F'$ meridionally compresses. 
The curve $\gamma$ is not separating, otherwise one component of $F'-n(\gamma)$ is a two-punctured torus $T$ with boundary $K_1$ and a copy of $\gamma$. Thus $\gamma$ and $K_1$ are homologous, implying that they have the same linking number with $K_2$, but $lk(K_1, K_2) \equiv$0 mod 2 and $lk(\gamma, K_2)= \pm 1$. 

Therefore,  after performing a meridian compression on $F'$ along $\gamma$, we obtain a 3-punctured torus $T$, with $\bd T= K_1 \cup m_1 \cup m_2$, where $m_1, m_2$ are meridians of $E(K_2)$.  Suppose $T$ is meriodionally incompressible, by Lemma \ref{KT}, $\vert Y(T) \vert=2$. Let $S_r$ be a regular level slightly above the first saddle. If $S\times [0,r] \cap T$ does not contain the boundaries $m_1, m_2$, using the curve $\beta$ generated by the first saddle we can argue as above that $T$ is either compressible or meriodionally compressible.

Suppose that $S\times [0,r] \cap T$ contains a boundary component, say $m_1$ and the first saddle occurs after $m_1$. Then the curve $\beta$ generated  by the first saddle, either separates or does not separate the points $S_r \cap K_2$, in either case there is a meridian compression for $T$.

We conclude that $T$ is meridionally compressible, moreover the boundary $\gamma$ of  the disk of meridian compression, is a non-separating curve in $T$, otherwise it would be a separating curve for $F'$.

Hence $T$ meriodionally compresses into a 5-punctured sphere $\Sigma$, which is essential and meridional incompressible in $E(K_2)$, the boundaries of $\Sigma$ consist of $K_1$ and four meridians of $E(K_2)$. Thus $\Sigma \subset S^3$ is a disk that intersects transversally $K_2$ in four points.
\end{proof}
\begin{figure}[!h]
\numberwithin{figure}{section}
\includegraphics[height=100mm]{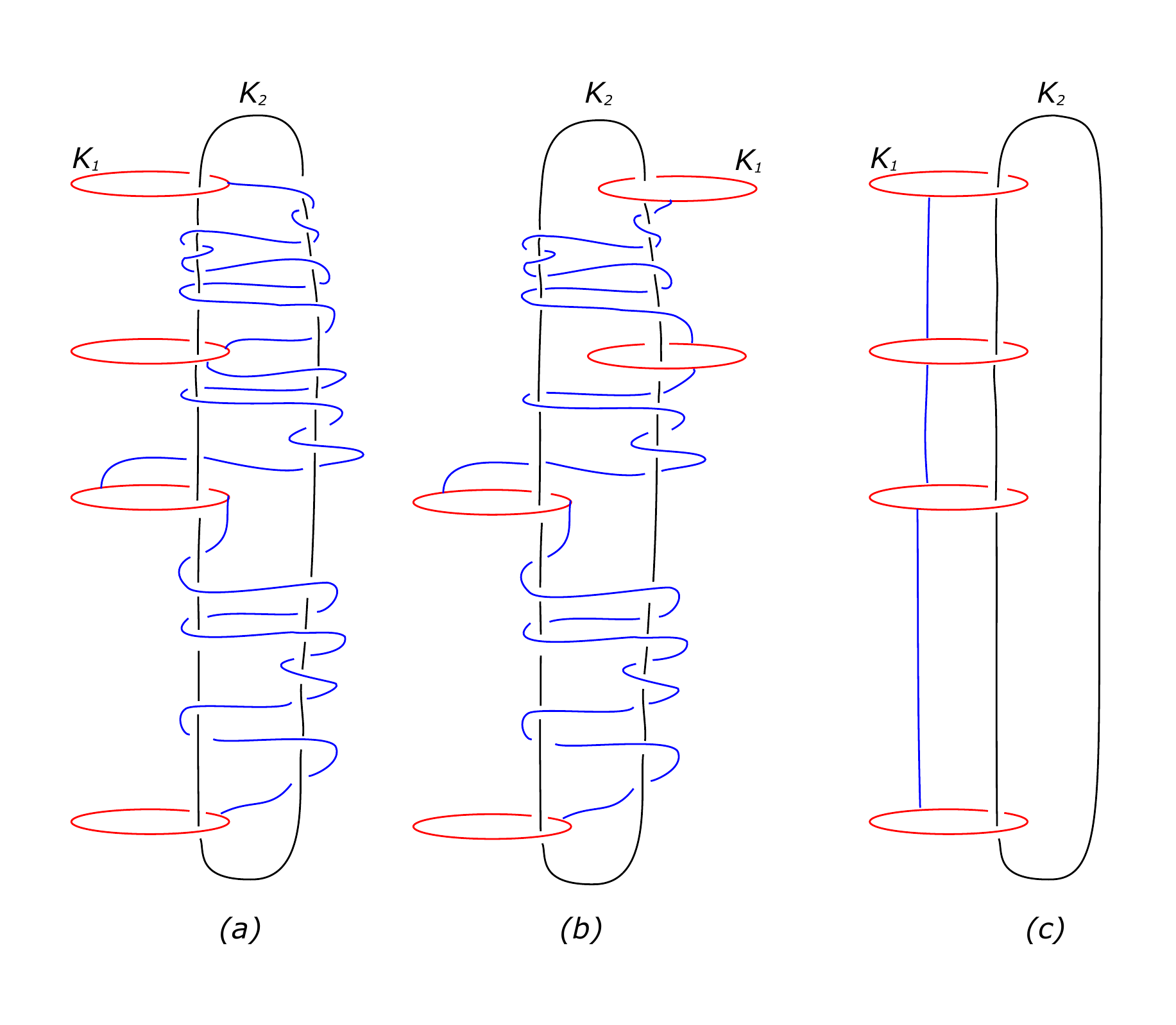}
\caption{The  arcs represent the cores of the bands along which the red disks are connected}
\label{2blink}
\end{figure}

\begin{figure}[!h]
\numberwithin{figure}{section}
\includegraphics[height=100mm]{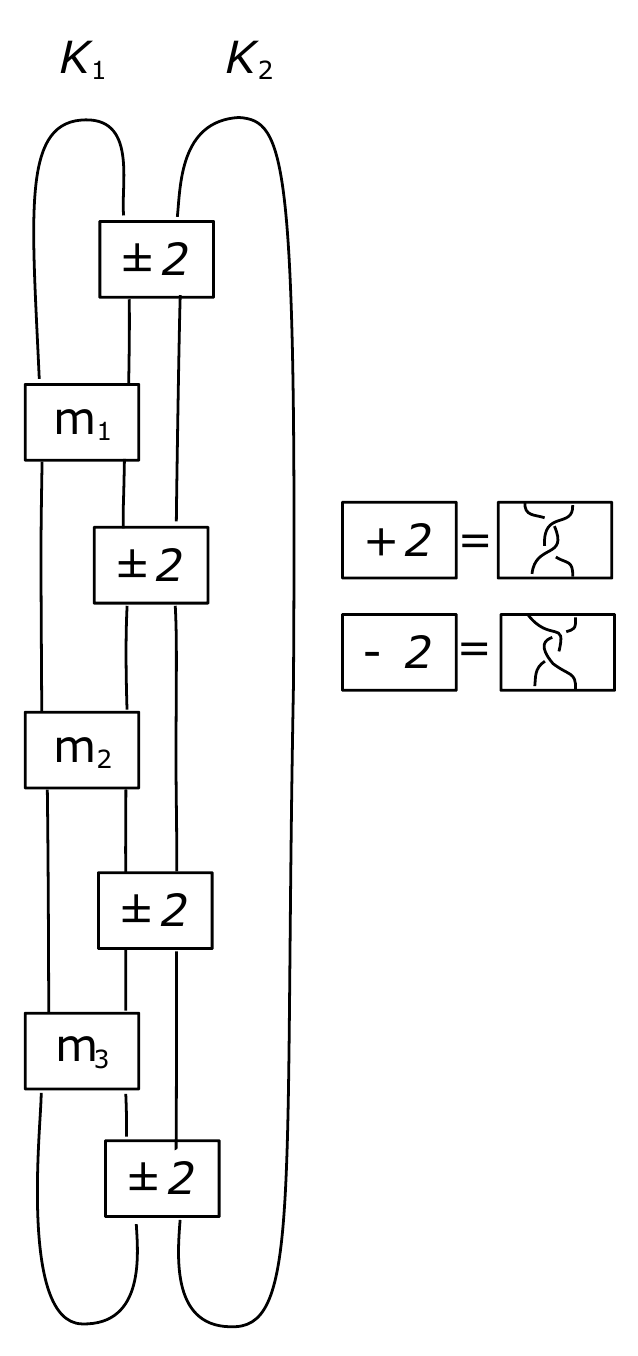}
\caption{}
\label{2blink1}
\end{figure}

\begin{proof}[Proof of Theorem~\ref{teoclasificacion}]
According to Lemma \ref{KT} we may assume that $\Sigma$ satisfies (M1)-(M4) and lies within the region $S\times I$ except for the cancelling disks $\Sigma_0$ and $\Sigma_1$ and $\vert Y(F)\vert=0$.
  
The disk $\Sigma$ can be constructed with four parallel disks transverse to $K_2$ joined by three descending bands disjoint from $K_2$, Figure \ref{2blink} (a), (b) show some possibilities. It is not hard to see that $\Sigma$ can be isotope to be as in Figure \ref{2blink} (c).

Thus the link $L= K_1\cup K_2$ is of the form shown in Figure \ref{2blink1}, where each $m_i$ is an integer and a box with a $m_i$ represents  $m_i$ crossings.

By an isotopy of $L$, we can assume that each box marked with $\pm 2$, is indeed $+2$. Using the fact that $lk(K_1, K_2)=0$ we find that  $m_1$ and $m_3$ have the same parity and $m_2$ has the opposite parity. Thus the link $L$ has associated the rational $\frac{\alpha}{\beta}$ given by the continued fractions $\frac{\alpha}{\beta}=[2, 2u, 2, 2v+1, 2, 2w,2]$ or $\frac{\alpha}{\beta}= [2, 2u+1, 2, 2v, 2, 2w+1,2]$, which is the result required by the Theorem.
\end{proof}


\begin{thebibliography}{99}
\bibitem{BZ} Gerhard Burde and Heiner Zieschang, \textit{Knots}, de Gruyter Studies in Mathematics (1985).

\bibitem{Eu} Mario Eudave-Mu\~noz, \textit{On nonsimple 3-manifolds and 2-handle addition}, Topology Appl. 55 (1994), no. 2, 131-152.

\bibitem{EuM} Mario Eudave-Mu\~noz, \textit{Incompressible surfaces in tunnel number one knot complements}, Topology Appl. 98 (1999), no.1-3, 167-189.

\bibitem{FH} W. Floyd, A. Hatcher, \textit{The space of incompressible surfaces in a 2-bridge link complement}, Trans. Amer. Math. Soc. 305 (1988), no. 2, 575-599.

\bibitem{GT} Hiroshi Goda and Masakazu Teragaito, \textit{Tunnel number one genus one non-simple knots}, Tokyo J. Math. 22 (1999), no. 1, 99-103.

\bibitem{M} Yves Mathieu,  \textit{Unknotting, knotting by twists disks and property $P$ for knots in $S^3$}. Knots 90 (Osaka 1990), 93-102, de Gruyter, Berlin, 1992.   

\bibitem {Ma} Hiroshi Matsuda, \textit{Genus one knots which admit $(1,1)$-decompositons}, Proc. Amer. Math. Soc. 130 (2002), no. 7, 2155-2163.

\bibitem{MS} Kanji Morimoto and Makoto Sakuma, \textit{On unknotting tunnels for knots}, Math. Ann. 289 (1991), no. 1, 143-167.

\bibitem {RV} E. Ram\'{i}rez-Losada and  L. G. Valdez-S\'anchez, \textit{Crosscap number two knots in $S^3$ with $(1,1)$ decompositions}, Bol. Soc. Mat. Mexicana 3 (2004), vol. 10, special issue, 451-465.

\bibitem {Sc} Martin Scharlemann, \textit{There are no unexpected tunnel number one knots of genus one}, Trans. Amer. Math. Soc. 356 (2004), 1385-1442.
\end{thebibliography}
\end{document}